\documentclass[a4paper,12pt,reqno]{extarticle}
\usepackage[applemac]{inputenc}
\usepackage{amsmath,amssymb,mathrsfs,amsthm}
\usepackage{geometry}
\usepackage{esint}

\usepackage{empheq}
\usepackage{enumerate,enumitem}
\usepackage{dsfont}
\usepackage[hidelinks]{hyperref}
\usepackage{times}

\theoremstyle{plain}
\newtheorem{lemma}{Lemma}[section]
\newtheorem{proposition}{Proposition}[section]
\newtheorem{corollary}{Corollary}[section]
\newtheorem{theorem}{Theorem}[section]

\theoremstyle{remark}
\newtheorem{remark}{Remark}[section]

\theoremstyle{definition}
\newtheorem{definition}{Definition}[section]
\newtheorem{claim}{Claim}[section]

\makeatletter

\@addtoreset{equation}{section}
\makeatother
\providecommand{\msc}[1]{{\small \textit{Mathematics Subject Classification---} #1}}
\providecommand{\keywords}[1]{{\small \textit{Keywords---} #1}}

\DeclareMathOperator*{\aslim}{a.s.\text{-}\,lim}
\DeclareMathOperator*{\esssup}{ess\,sup}

\DeclareMathOperator*{\card}{\#}

\DeclareMathOperator*{\transp}{{}^t\!}
\DeclareMathOperator*{\spt}{spt}

\DeclareMathOperator{\id}{id}
\DeclareMathOperator{\curl}{curl}
\DeclareMathOperator{\nb}{\nabla}
\DeclareMathOperator{\nbp}{\nabla^\perp\!}

\def\onedot{$\mathsurround-0pt\ldotp$}
\def\cddot{
  \mathbin{\vcenter{\baselineskip.67ex
    \hbox{\onedot}\hbox{\onedot}}%
  }}%

\newcommand{\us}{\ensuremath{u'}}
\newcommand{\Omegas}{\ensuremath{\Omega'}}
\newcommand{\PPs}{\ensuremath{\mathfrak P'}}
\newcommand{\Ps}{\ensuremath{\mathbb P'}}
\newcommand{\Es}{\ensuremath{\mathbb E'}}
\newcommand{\Ws}{\ensuremath{W'}}
\newcommand{\Fs}{\ensuremath{\mathcal F'}}
\newcommand{\zetas}{\ensuremath{{\zeta'}}}
\newcommand{\uss}{\ensuremath{u''}}
\newcommand{\zetass}{\ensuremath{{\zeta''}}}
\newcommand{\Wss}{\ensuremath{W''}}
\newcommand{\N}{\ensuremath{\mathbb N}}
\newcommand{\R}{\ensuremath{\mathbb R}}
\newcommand{\T}{\ensuremath{{\mathbb T^2}}}
\renewcommand{\S}{\ensuremath{\mathbb S^2}}
\newcommand{\PP}{\ensuremath{\mathfrak P}}
\renewcommand{\P}{\ensuremath{\mathbb P}}
\newcommand{\E}{\ensuremath{\mathbb E}}
\newcommand{\law}{\ensuremath{\mathcal L}}
\newcommand{\ps}[3]{\ensuremath{\langle #1,#2\rangle_{#3}}}
\newcommand{\n}[2]{\ensuremath{|#1|_{#2}}}
\newcommand{\nn}[2]{\ensuremath{\|#1\|_{#2}}}
\DeclareMathOperator{\divergence}{div}
\renewcommand{\div}{\ensuremath{\divergence}}
\newcommand{\Dp}[2]{\ensuremath{\frac{\partial #1}{\partial #2}}}

\newcommand{\cc}{\ensuremath{\boldsymbol\phi}}
\newcommand{\hs}{\ensuremath{\mathbb L_2}}
\newcommand{\cov}[1]{\ensuremath{\mathfrak{C}(#1)}}
\newcommand{\sing}{\ensuremath{\mathrm{Sg}}}
\newcommand{\uu}{\ensuremath{{\hat u}}}
\renewcommand{\tt}{{\ensuremath{{\mathscr T_u}}}}
\renewcommand{\d}{\ensuremath{\hspace{0.05em}\mathrm{d}}}      

\title{Struwe-like solutions for the Stochastic Harmonic Map Flow\thanks{Second version, November 2018.}}

\author{Antoine Hocquet
\thanks{Email: antoine.hocquet@wanadoo.fr, Tel.: +49 30 314-25728, Fax: +49 30 314-25191}
  \thanks{This work was done while the author was affiliated at the CMAP, \'Ecole Polytechnique, CNRS, Universit\'e Paris Saclay, 91128, Palaiseau, France.}
\\{\small Technische Universit\"at Berlin, Instit\"ut f\"ur Mathematik, Strasse des 17. Juni 136,}
\\{\small Berlin-Charlottenburg, Germany.}
}
\date{}

\begin{document}

\maketitle
\indent\keywords{Stochastic partial differential equation - Harmonic Maps - Nonlinear parabolic equations\\}
\indent\msc{60H15 (35R60), 58E20, 35K55}

\begin{abstract}
We give new results on the well-posedness of the two-dimensional Stochastic Harmonic Map flow, whose study is motivated by the Landau-Lifshitz-Gilbert model for thermal fluctuations in micromagnetics. We first construct strong solutions that belong locally to the spaces $C([s,t);H^1)\cap L^2([s,t);H^2)$, $0\leq s<t\leq T$. It that sense, these maps are a counterpart of the so-called ``Struwe solutions'' of the deterministic model. We then provide a natural criterion of uniqueness that extends A.\ Freire's Theorem to the stochastic case. Both results are obtained under the condition that the noise term has a trace-class covariance in space.
\end{abstract}

\tableofcontents

\section{Introduction}
\subsection{Motivations}
In this paper, we are interested in the existence, uniqueness and regularity
of the parabolic stochastic partial differential equation
\begin{equation}\label{SPDE}
\left\{\begin{aligned}
&\frac{\partial u}{\partial t}-\Delta u=u|\nabla u|^2 + u\times \xi \quad\text{in}\enskip (0,T]\times \T\,,
\\
&u(0)=u_0 \quad\text{in}\enskip  \T\,,
\end{aligned}\right.
\end{equation}
(Stratonovitch sense)
whose unknown $u:\Omega\times[0,T]\times\T\to\S\subset\R^3$ takes values in the unit sphere $\S\equiv\{x\in\R^3:|x|=1\},$ and where the initial datum $u_0$ belongs to the critical space $H^1.$
We also assume regularity in space for $\xi =\xi (\omega ,t,x),$ which is
the time derivative of a Wiener process with finite trace class covariance in $H^1$.
We will first construct the counterpart of the so-called ``Struwe solutions'' in the presence of noise (Theorem \ref{thm:struwe_sol}).
Then, a similar result as that of Freire's uniqueness Theorem will be given, providing a natural criterion of uniqueness leading to the solution obtained above (Theorem \ref{thm:uniqueness}).

The deterministic equation corresponding to \eqref{SPDE} has been first studied in the early sixties by J.\ Eells and J.H.\ Sampson \cite{eells1964harmonic}, in order to build harmonic maps from a general manifold (which here is simply the two-dimensional torus $\T$) onto another, typically a unit sphere.
Trying to find harmonic maps $u:M\to N$ between two manifolds provides an important example of a variational problem occuring in the context of non-flat metrics. It echoes several physical models,
such as liquid crystals \cite{coron1991nematics}, or W.F.\ Brown's theory for continuous micromagnetics \cite{brown1963micromagnetics}. Their common feature is the necessity for ground states to minimize the functional $E:=\int_M|\nabla u|^2\d x$, under the pointwise constraint $u(x)\in N$ a.e.

Whether there exists or not a harmonic map, within the homotopy class of a given smooth map $\varphi :M\to N$ is by itself an important topic for geometers.
To answer that question, the approach initiated by J.\ Eells and J.H.\ Sampson consists in adding a time variable to the unknown, and then studying the Heat flow associated to $E$, namely 
\begin{equation}\label{HMF}\tag{HMF}
\Dp{u}{t}-\Delta_M u=u|\nabla u|^2\, ,\enskip t\geq 0\enskip ,
\\
\enskip u|_{\{0\}\times M}=\varphi\,,
\end{equation}
where in view of the applications we let here $N:=\S.$
The next step is to show convergence of the solution as $t\to\infty$, towards an harmonic map. This follows from asymptotic estimates, yielding finally a solution to the problem, see e.g.\ \cite{eells1964harmonic,eells1978report,eells1988another}.

\paragraph{Struwe-like solutions}
Unfortunately, the latter method fails unless the target manifold has non-positive sectional curvature, a somewhat restrictive hypothesis.
If $M$ denotes a surface, M.\ Struwe has shown in \cite{struwe1985evolution} that \eqref{HMF} admits a solution $u$ such that
\begin{enumerate}[label=-]
\item $u$ fulfills \eqref{HMF} in the sense of distributions;
\item $u$ is a classical, smooth solution, with the exception of finitely many  points $(t_{i},x_{i}^k), k\leq K_i, i\leq I$.
\end{enumerate}
The latter map is the one that we might refer to as the ``Struwe solution''. As will be shown below, the Struwe solution has a natural counterpart in the presence of noise.
Note that existence of a weak martingale solution has been provided in \cite{brzezniak2013weak}, where the authors are able to deal with a three dimensional domain. This is done via finite-dimensional Galerkin approximations and uniform energy bounds on the corresponding family.

Our approach here is different, in the sense that we work at the level of a regularized stochastic PDE, but still infinite-dimensional. We will obtain strong solutions by taking sufficiently smooth initial data, as well as a sufficiently correlated noise. A similar tightness argument as above will be used thanks to a priori estimates, which are justified by the fact that the approximations are regular enough to yield an It\^o formula.
This is somewhat faithful to Struwe's original approach, with the difference that for stochastic PDEs, existence and uniqueness have more varied aspects.
We will indeed see that this method yields strong solutions in the probabilistic sense. Moreover, the justification of the It\^o formula requires here a bootstrap argument. This method suffers the fact that, no matter which state space $X$ we choose for a solution $t\mapsto u(t)$, we will always have $u\notin C^{1/2}(0,T;X)$, so that adapting the deterministic tools may be involved.
We will circumvent this problem by using the ideas presented in \cite{debussche2015regularity}. However, an additional difficulty here is the polynomial nonlinearity $u|\nabla u|^2$ which, to the best of the author's knowwledge, has not been treated so far.

\paragraph{Criticality and uniqueness}
It turns out that the Struwe solution is unique in the class of solutions depending continuously of the initial data $\varphi$ in $H^1$, locally in time. That is: for some $t_1>0$ and for every $t<t_1$, then $u_n(0)\to\varphi$ in $H^1$ implies $u_n\to u$ in $C([0,t];H^1)\cap L^2([0,t];H^2)$.
At the singular points (more precisely when $t\nearrow t_i$ and $x\to x_i^k$), we observe peaks in the energy density $x\mapsto|\nabla u(t,x)|^2$ that are called in the literature ``forward bubblings''. An amount of energy is then released: although $u$ stays in $H^1$ for all times,
the inequality 
\[E(t_i)\leq\liminf_{t\nearrow t_i}E(u(t))\]
is strict.
The key ingredient for the proof of that result is a sharp interpolation inequality that permits to control the nonlinear term through the energy bound -- see Proposition \ref{pro:interp2} below. Such an estimate is of course, specific to the dimension two.

It has been widely observed (see e.g.\ \cite{helein1996applications} for an overview of the subject) that the harmonic map problem
\begin{equation}\label{HM}
-\Delta_M u=u|\nabla u|^2\,,\quad\text{where}\enskip u=u(x)\,,
\end{equation}
has specific features in dimension two, as for instance a theorem due to F.\ Helein \cite{helein1990regularite} states that any 2D weakly harmonic map (that is in the sense of distributions) is actually harmonic in the classical sense.
Concerning this time \eqref{HMF}, the associated natural energy $E$ fulfills the a priori bound:
\begin{equation}\label{apriori_energy}
E(t)-E(0)+\iint_{[0,t]\times M}|\Delta u+u|\nabla u|^2|^2=0\, ,
\end{equation}
which barely fails to give well-posedness of the flow.
Indeed, \eqref{apriori_energy} yields that
the nonlinearity $u|\nabla u|^2$ belongs at each time to $L^1$, which in dimension two ``hardly differs from $H^{-1}$'' in the sense that $L^p\hookrightarrow H^{-1}$ is always true unless $p=1$.
This small difference turns out to be important: if for some reason we could obtain that $u\in C (H^{-1})$ locally in time, then standard results on heat equations would yield well-posedness.
In the time-independent case, this criticality is outpassed in the proof of the latter Helein's theorem, by slightly increasing regularity from the symmetries of the associated variational problem. More precisely,
noticing that the nonlinearity has the particular form
$u^i|\nabla u|^2=A\cddot\nabla u^i\equiv\sum_{k,j} A^{i,j}_k \partial _ku^i$, $i\leq3$, with $\div A=0$ (this latter conservation law stems from the fact that $u$ is harmonic), then
classical results on the decomposition of 2D vector fields, imply the existence of $\beta \in (\R^3)^{\otimes2}$ such that
\begin{equation}\label{poisson}
u^i|\nabla u|^2=\Dp{\beta ^i}{x_1}\cdot \Dp{u}{x_2}-\Dp{\beta^i }{x_2}\cdot \Dp{u}{x_1}=:\{\beta ,u\}^i\,,\quad i=1,2,3\,.
\end{equation}
Now, Wente's inequality \cite{wente1969existence}, which can be seen as a two-dimensional analogue of the more celebrated 3D ``$\div-\curl$ Lemma'' (see \cite{murat1981compacite,tartar1983compensated}), states that the quantity $\{\beta,u\}$ has the rather unexpected property of being continuous in $H^{-1}$ with respect to the weak topology of $H^1$, although being nonlinear.

In the non-stationnary setting we can still write the latter decomposition, with the difference that $\div A(t)\neq0$, and therefore we do not have \eqref{poisson}. Nevertheless, it is still possible to treat apart some additional non-divergence-free term. This approach turns out to be essential in the proof of the following uniqueness result:
\begin{theorem}[\cite{freire1995uniqueness}]
Any weak solution $u$ of \eqref{HMF} such that $E(t)\equiv\frac12\int_{\{t\}\times M}|\nabla u|^2$ is non-increasing with $t\geq 0$ is the Struwe solution.
\end{theorem}
\begin{remark}
The notation $\int_{\{t\}\times M}$ means that we integrate the trace of $|\nabla u|^2$ onto $\{t\}\times M$, which is defined \emph{for every $t\in[0,T]$}. Note that P.~Topping has given an example where $E(t)\leq E(s)$ for a.e.\ $t\geq s$ and yet $u$ is not the Struwe solution (see \cite{topping2002reverse}).
In the following we will systematically assume that $E(t)$ corresponds to the latter integral.
\end{remark}
The proof of this theorem exploits the fact that although $\div A(t)\neq0$, we can write $A(t)=\nabla\alpha+\nabla^\perp \beta $, where the second term is a divergence-free tensor, so that by Wente's theorem $II\equiv\nbp \beta \cddot\nabla u$ can be writen as the sum of a small $C(H^{-1})$ term, plus a regular remainder. The additional time-regularity is obtained as a consequence of the monotonicity of $E$. On the other hand, the first term $I\equiv\nb\alpha\cddot\nb u$ is controlled by the estimate
\[\iint_{[0,T]\times M}|\Delta \alpha |^2\lesssim \iint_{[0,T]\times M}|\Delta u +u|\nabla u|^2|^2\, ,\]
where due to the a priori estimate \eqref{apriori_energy}, the r.h.s.\ in the latter bound is finite for every ``reasonable'' definition of a weak solution to \eqref{HMF}.
Denoting by $u_j$, $j=1,2$ two solutions of the problem,
uniqueness is provided by linearizing the equation for $u_1-u_2$, around the solution that corresponds to that constructed in \cite{struwe1985evolution}. This proof has however the disadvantage of appealing to M.\ Struwe's existence part. Here our uniqueness theorem, namely Thm.\ \ref{thm:uniqueness} uses a different approach, though related through technical aspects. This is new even in the deterministic setting, where the new proof can be computed simply by letting $W\leftarrow0$.

\paragraph{The Harmonic Map Flow perturbed by Gaussian noise}
As already pointed out, the minimization problem associated to $E$ relates the theory of micromagnetism where admissible configurations of the magnetization of a ferromagnetic domain $M$, $1\leq\dim M\leq3$, are the minimizers of the Dirichlet energy.
Out of equilibrium, the dynamics of the magnetization $u:[0,T]\times M\to \S$ is governed by Landau-Lifshitz-Gilbert equation \cite{landau1935theory,gilbert1955lagrangian,gilbert2004phenomenological}
\begin{equation}\label{LLG}
\Dp{u}{t}=\gamma u\times H_\mathrm{eff}-u\times(u\times H_\mathrm{eff})\, ,
\end{equation}
where $H_\mathrm{eff}:=-\nabla E(u)$, and $\gamma \in\R$ is the gyromagnetic ratio.
The geometrical constraint on the magnetization, namely ``$u(x)\in\S,\forall x$'', stems from the fact that below the so-called Curie point, the value of $|u(x)|$ depends on the temperature only.
If we stay below this level, but at a sufficiently high temperature so that thermal effects are no longer negligible, then
fluctuations of the effective field are such that
\begin{equation}\label{Heff}
H_\mathrm{eff}=\Delta u+ \xi 
\end{equation}
where $\xi =\xi (t,x)$ denotes Gaussian white noise, see \cite{brown1963thermal,Berkov}. In the framework of stochastic equations in infinite dimensions, this term is classically constructed as the formal sum
\begin{equation}\label{white}
\xi (t,x):=\frac{\d W}{\d t}(t)(x)\equiv\sum_{\ell \in\N}\frac{\d B_\ell}{\d t} (t)\cc e_\ell(x) \, ,
\end{equation}
that is: $\xi $ is the time-derivative of the $\cc\cc^*$-Wiener process $W(t):=\sum B_\ell(t)\cc e_\ell.$
Here we denote by $(e_\ell )$ a fixed orthonormal system of $L^2(\T;\R^3)$ and by $(B_\ell )$ an i.i.d.\ family of real-valued Brownian motions. We are given a bounded linear operator $\cc:L^2\to L^2,$ ``measuring the spatial corellation'' of $W$ through the formula:
\[
\E\left[\ps{W(t)}{f}{}\ps{W(s)}{g}{}\right]=\min(s,t)\ps{\cc f}{\cc g}{L^2}\,,\quad\forall f,g\in L^2(\T;\R^3)\,.
\]
Let us mention that solvability of \eqref{LLG-Torus} in the case where $\xi$ is white in time \textit{and space} (that is when $\cc=\id$) is not a problem that we adress here.
In dimension two, the cylindrical Wiener process is not better than $\cap_{\epsilon >0}H^{-1-\epsilon }$ in space, which matches the regularity of the nonlinearity. Hence \eqref{LLG} is critical in the sense given in 
\cite{hairer2014theory}, so that the theory of regularity structures does not apply.
We will assume throughout the paper that $\cc $ is Hilbert-Schmidt from $L^2$ to $H^1$, which is needed to make sense of the energy.

Now, because of the norm constraint we have the vectorial identity:
$-u\times(u\times\Delta u)=\Delta u+u|\nabla u|^2$, so that setting for simplicity
\begin{equation}\label{overdamped}
\gamma =0\, ,
\end{equation}
while forgetting the contribution $-u\times (u\times \circ\d W)$ to keep $u\times\d W$ only, we end up with  the Stratonovitch equation:
\begin{equation}\label{LLG-Torus}
\left\{\begin{aligned}
&\d u=(\Delta u+u|\nabla u|^2)\d t+u\times \circ\d W\, ,&&\text{on}\enskip  \Omega\times(0,T]\times\T\\
&u|_{t=0}=u_0 \,,&&\text{on}\enskip \Omega\times\T\, ,
\end{aligned}\right.
\end{equation}
where $\sum_{\ell }|\cc e_\ell |_{H^1}^2<\infty$, and $u_0\in H^1$.

The parallel between \eqref{HMF} and the (deterministic) equation \eqref{LLG} has provided interesting insights, see e.g.\ \cite{alouges1992global,guo1993landau,harpes2004uniqueness,guo2008landau}. Our results below could be stated in presence of a gyromagnetic term in \eqref{LLG}, provided however, that local smooth solutions exist for regular data $(u(0),\phi )$ (see \cite{hocquet2015landau}). Unfortunately, the method presented in sec.\ \ref{sec:local_solv} below to obtain local solvability ceases to work for the case $\gamma\neq0$. We hope to successfully come back to this question in a forthcoming work.
To simplify the presentation, we will restrict our attention to the case where $M$ equals $\T$, the two-dimensional torus. Nevertheless our results could be adapted to the case of a general surface with boundary, endowed with a Riemannian metric (see e.g.\ \cite{kung1989heat} for a treatment of the deterministic case).

Note that the a priori estimate on the energy writes this time:
\begin{equation}\label{apriori2}
E(t)-E(0)+\iint_{[0,t]\times M}|\Delta u+u|\nabla u|^2|^2=C_\phi t+M(t)\, ,
\end{equation}
where $M(t)$ is a martingale, and $C_\phi:= |\nabla \phi|_{\hs}^2$ (see the notations below). Although in this context the energy cannot decrease pathwisely, we see that regular solutions of \eqref{LLG-Torus} must be such that
the quantity
\[
\mathscr G(t):=E(t)-C_\phi t
\]
defines a supermartingale with respect to $(\mathcal F_t)$.
This property turns out to be the ``correct'' stochastic counterpart as that of A.\ Freire's criteria that the energy decreases. This is somehow reminiscent to the notion of ``energy solution'' given for the 3D Navier-Stokes equation in \cite{flandoli2006markov}.
\subsection{Notation and settings}
The letters $C,C',\tilde C$ etc.\ are used to denote positive constants that may change from line to line.
When we want to emphasize their dependency with respect to some parameters $\phi ,\psi ,\dots,$ we write for instance $C(\phi ,\psi ,\dots)$ instead.
Given $a,b\in\R,$ the notation $a\wedge b$ stands for the minimum $\min(a,b).$
The notation ``$\lim_{t\nearrow t_*}f(t)$'' means that the limit of the quantity $f$ is taken over any increasing sequence $t_k\to t_*$ such that $t_k<t$ for all $k\in\N$ (and similarly for ``$\lim_{t\searrow t_*}$'').

For $d\in\N$, $\alpha ,p\in\R$,
the usual Lebesgue and Sobolev-Slobodeckij spaces $L^2(\T;\R^d),$ $L^p(\T;\R^d)$, $W^{\alpha ,p}(\T;\R^d)$, $H^{\alpha }(\T;\R^d)\equiv W^{\alpha ,2}(\T;\R^d)$, etc.\ are occasionally abbreviated as $L^2$, $L^p$, $W^{\alpha ,p}$, $H^{\alpha }$. We write the corresponding norms as $\n{\cdot }{L^2},\n{\cdot }{L^p},\n{\cdot }{W^{\alpha ,p}},\n{\cdot }{H^\alpha }$. The $L^2$ inner product will be denoted by $\langle\cdot  ,\cdot \rangle$, namely
\begin{equation}\label{nota:inner}
\langle f,g\rangle:=\int_{\T}f(x)g(x)\d x\,,\quad f,g\in L^2\,.
\end{equation}
Throughout the paper, we work on a compact time interval $[0,T]$ where, except in the proof of Theorem \ref{thm:localSolv},
we consider $0<T<\infty$ as a fixed parameter.
When refering to space-time elements, we write \[f\in L^q(0,T;X)\] to say that $f\in L^q([0,T];X),$ that is
$\int_0^T |f(t)|_{X}^{q}\d t<\infty$, $X$ being any of the Banach spaces above. The associated norms are sometimes abbreviated as $\nn{\cdot }{L^q(0,T;X)}$, or simply $\nn{\cdot }{L^q(X)}$, when the interval on which we integrate is clear from the context.
When $\alpha p\geq 1,$ we will denote by $W^{\alpha ,p}_0(0,T;X)$ the space corresponding to those $f\in W^{\alpha,p} (0,T;X),$ such that $f(0)=0$ in the sense of traces, and similarly for $H^\alpha _0(0,T;X).$

The notation 
\begin{equation}\label{nota:L2loc}
f\in L^2_\mathrm{loc}([\tau_1,\tau_2);H^2)
\end{equation}
means that $\nn{f}{L^2(K;H^2)}<\infty$ for every compact interval $K\subset[\tau _1,\tau _2).$

If $X,Y$ are Banach spaces,
we denote by $\mathscr L(X,Y)$ the space of bounded linear operators. We denote by $\gamma(Y)$ the space of $\gamma $-radonifying operators from the Hilbert space $L^2$ onto $Y$, that is
\[T\in \gamma(Y)\quad\Leftrightarrow\quad|T|^2_{\gamma(Y)}\equiv\int_{\tilde\Omega}\left|\sum\nolimits_{\ell \in\N} \gamma _\ell(\tilde\omega ) Te_\ell \right|^2_{Y}\mathbb{\tilde P}(\d \tilde\omega )<\infty\,,\]
for all orthonormal systems $(e_\ell )\subset L^2$, and for $(\gamma _\ell )$ an i.i.d.\ family of $\mathscr N(0,1)$ random variables defined on some probability space $(\tilde\Omega,\mathcal {\tilde F},\mathbb{\tilde P})$.
When $Y$ is a Hilbert space, $\gamma (Y)$ corresponds to the class of Hilbert-Schmidt operators from $L^2$ onto $Y$ and it will be denoted by \[
\hs(Y),\quad\text{or simply}\quad\hs\quad\text{if}\enskip Y=L^2
\,.\]
For $s\in\R$ we also use the abbreviation 
\[\hs^{s}:=\hs(H^s)\,.\]

In the whole paper,
we fix a stochastic basis $\mathfrak P=(\Omega,\mathcal{F},\P,(\mathcal{F}_t)_{t\in[0,T]}, W)$, that is a filtered probability space, together with an $L^2(\T;\R^3)$-valued Wiener process $W$ with respect to a right-continuous filtration $(\mathcal{F}_t)_{t\in[0,T]}$.
We assume that $W(1)$ has spatial covariance  $\cc \cc ^*$, where to simplify the presentation, we make the assumption that the correlation is ``isotropic'', namely there exists an Hilbert-Schmidt operator
$\phi :L^2(\T;\R)\to L^2(\T;\R)$
such that
$\cc:L^2(\T;\R^3)\to H^1(\T;\R^3)$, is the operator given by
\begin{equation}\label{isotropic}
(f_1,f_2,f_3)\mapsto \cc f:=(\phi f_1,\phi f_2,\phi f_3)\,.
\end{equation}
We assume in addition that $W$ is given by the sum
\begin{equation}\label{nota:Wphi}
W(t):=\sum_{\ell \in\N}B _\ell (t)\cc e_\ell \, ,
\end{equation}
where
$e_\ell $ and $B_\ell  \,,\ell \in\N$ are as in \eqref{white}.

Assuming that $u$ is solution to \eqref{LLG-Torus} and that $\Psi\in C^1(L^2;\hs),$
the Stratonovitch product is given by the rule:
$\Psi(u)\circ \d W = \Psi(u)\d W +1/2\sum_{\ell \in\N} \Psi'(u) [u\times\cc e_\ell ]\cc e_\ell \d t$,
provided the right hand side is convergent. For $\cc $ as in \eqref{isotropic}, this yields for instance that $u\times\circ\d W=u\times\d W+ F_\phi u\d t ,$ where in the sequel we shall denote by
\begin{equation}\label{F_phi}
F_\phi (x):=\sum_{\ell \in\N}-(\phi e_\ell (x))^2\,,\quad x\in\T\,.
\end{equation}

\subsection{Notion of solution and main Results}
We will make use of two different notions of solution for \eqref{LLG-Torus}.  The ``strong solutions'' are both strong in PDE and Probabilistic sense, and 
yield the correct framework to locally describe the so-called Struwe solution
constructed in the theorem below.
\begin{definition}[local strong solutions]\label{def:StrongSol}
Assume that a stochastic basis $\mathfrak{P}$ is given, that
$0\leq\tau _1<\tau _2\leq T$ are stopping times, and that $u_1\in L^2(\Omega;H^1)$ is $\mathcal F_{\tau _1}$ measurable.
Given a progressively measurable process $u:\Omega\times[0,T]\to H^1$,
we say that $(u;\tau _1,\tau _2)$,
is a \textit{local strong solution} of \eqref{LLG-Torus} on $[\tau _1,\tau _2)$,
with initial datum $u_1$ if the following conditions are fulfilled:
\begin{enumerate}[label=(\roman*)]
\item for $\P\otimes \d t\otimes \d x$ a.e.\ $(\omega ,t,x)$ with $\tau _1(\omega )\leq t<\tau _2(\omega )$, there holds
\begin{equation}\label{sphere_constraint}
|u(\omega ,t,x)|_{\R^3}=1\,;
\end{equation}
\item $\P\text{-a.s.},$ the process $u$ has paths in $C([\tau _1,\tau _2);H^1)\cap L^2_{\mathrm{loc}}([\tau _1,\tau _2);H^2)$;
\item $\P\text{-a.s.},$ for $t\in[\tau _1,\tau _2)\,$:
 $  u(t)-u_1=\int_{\tau _1}^{t}(\Delta u+u|\nabla u|^2+F_\phi u )\d t +\int_{\tau _1}^tu\times \d W$
in the sense of Bochner, respectively It\^o integral in $L^2$.
\end{enumerate}
\end{definition}
Our main results are stated in the following two theorems.
Here we let $H^1(\T;\S):=H^1(\T;\R^3)\cap\{v:v(x)\in\S\text{ a.e.\null}\}$.
\begin{theorem}
\label{thm:struwe_sol}
Let $\cc\in\hs^1$ such that \eqref{isotropic} holds and $W$ as in \eqref{nota:Wphi}.
 For all $T>0$, and $u_0\in H^1(\T;\S)$, there exist $u\in L^\infty(\Omega ;C(0,T;L^2))\cap L^2(\Omega;L^\infty(0,T;H^1))$ and a finite sequence of stopping times
 $\vartheta ^0\equiv0<\vartheta^1< \vartheta^2< \dots<\vartheta^{J(\omega )}$, with
\begin{equation}\label{thm_1}
\P\left(\vartheta^J= T\right)=1\,,
\end{equation}
such that for each $0\leq j\leq J-1,$
$(u|_{[\vartheta ^j,\vartheta ^{j+1})};\vartheta^j,\vartheta^{j+1})$ 
is a local strong solution of \eqref{LLG-Torus} with initial datum $f^j$ at $t=\vartheta^j$, where $f^j$ is uniquely determined by 
\begin{multline}\label{thm_2}
u(\zeta _{k}^j)\to f^j\,,\enskip\P\text{-a.s.},\,\text{weakly in}\enskip H^1(\T;\R^3)\,,\\
\text{for every sequence of stopping times}\enskip \zeta _{k}^j \nearrow \vartheta^j.
\end{multline}

Moreover, there is a universal constant $\epsilon_1>0$ such that on $\{J>1\},$
letting $j\in\{0,\dots,J-2\},$ there exists a random set $\sing^{j+1}\subset \T$ with $1\leq \card\sing^{j+1}<\infty$
with the property that for $x\in\sing^{j+1}:$
\begin{equation}
\label{thm_4}
\lim_{t\nearrow \vartheta^{j+1}}\int_{B(x,\varrho)}|\nabla u(t,y)|^2\d y\geq \epsilon _1\,,
\end{equation}
independently of $\varrho>0.$
Furthermore, it holds
\begin{equation}
\label{thm_5}
|\nabla u(\vartheta^{j+1})|\leq \liminf_{t\nearrow \vartheta^{j+1}}|\nabla u(t)|^2_{L^2}
-(\card\sing^{j+1})\epsilon _1\,.
\end{equation} 
\end{theorem}
\begin{remark}
Using \eqref{ineq:Gronwall} together with \eqref{ineq:interp_loc}, the proof below shows that the solutions constructed in Theorem \ref{thm:struwe_sol} are unique in their class.
Namely, if we are given a martingale solution $v$, fulfilling the property:
\begin{multline*}
``\text{there exists a \emph{positive} stopping time  $\tau $ such that }
\\
\inf_{\varrho>0}\sup_{0\leq t\leq\tau}\sup_{x\in\T}\int_{B(x,\varrho)}|\nabla v(t,y)|^2\d y=0\,'',
\end{multline*}
then we have $v|_{[0,\tau ]}=u|_{[0,\tau ]}.$
\end{remark}

\begin{remark}
Concerning the Dirichlet problem, the existence of finite-time blowing-up solutions has been provided in \cite{hocquet2016finite}, but further degeneracy assumptions on the noise have to be made.

In the general case, it remains an open problem to build such singular solutions.
Moreover, unstability results on the deterministic equation suggest a possible ``regularization by noise'' phenomenon. It is indeed expected that the sequence $\{\vartheta^j:0\leq j\leq J\}$ is always the set $\{0,T\}$, see \cite{merle2011blow} and the closing remarks in \cite{hocquet2016finite}.
\end{remark}

The second notion of solution that we need to introduce corresponds to that of \cite[chap.~8]{DPZ}. It is also motivated by the results obtained in \cite{brzezniak2013weak,banas2013convergent,hocquet2015landau}.
\begin{definition}[weak martingale solution]
\label{def:martingale_solution}
A \emph{weak martingale solution} for \eqref{LLG-Torus} is a couple $(\PPs,\us)$, $\PPs$ being a stochastic basis $(\Omegas,\Fs,\Ps;(\Fs_t)_{t\in[0,T]}, \Ws)\,$ where
$ \Ws$ has covariance $\cc\cc^*\,$, whereas $ \us:\Omegas\times[0,T]\to H^1$ is an $(\Fs_t)$-progressively measurable process
satisfying the following:
\begin{enumerate}[label=(\roman*')]
\item $|\us|_{\R^3}=1\,$, $\Ps\otimes\d t\otimes\d x$-a.e.\null;
\item\label{item_2}
$\us \in C(0,T;L^2)$, $\Ps$-a.s.\ and \[\Es\left[\esssup_{t\in[0,T]}|\nabla \us(t)|_{L^2}^2 + \int_0^T|\Delta \us+\us|\nabla \us|^2|_{L^2}^2\d t\right]<\infty;\]
\item $ \us$ satisfies
$\us(t)=u_0+\int_0^t\left(\Delta  \us +  \us|\nabla  \us|^2 +F_\phi\us\right) \d r+\int_0^t  \us(r)\times \d \Ws$ for all $t\in[0,T]$, $\Ps$-a.s.\ ,where the first integral is the Bochner integral, and the second is the It\^{o} integral, in the space $L^2\,$.
\end{enumerate}
\end{definition}
\begin{remark}\label{rem:struwe_map}
By ``gluing together'' the local solutions constructed in Theorem \ref{thm:struwe_sol}, straightforward computations show that the resulting map is in fact a (martingale) solution over the whole interval $[0,T]$.
This solution will be referred to as ``the Struwe solution''.
\end{remark}
\begin{theorem}[Uniqueness criterion]
\label{thm:uniqueness}
Let $u$ denote a martingale solution of \eqref{LLG-Torus} in the sense given in Definition \ref{def:martingale_solution}, with an associated $\phi \in\hs^{1}$.

Assume that $u$ is such that $\mathscr G$, its energy corrected for the mean injection rate from $W$, namely the process
\[
\mathscr G(t):=\frac{1}{2}|\nabla u(t)|_{L^2}^2 -|\nabla\phi|_{\hs}^2t\,,\quad t\in[0,T]
\]
(here $u(t)$ denotes the trace onto the slices $\{t\}\times\T$),
is a \emph{supermartingale} with respect to $(\mathcal{F}_t)\,$.
Then $u$ is the Struwe solution (see remark \ref{rem:struwe_map}).
\end{theorem}
\paragraph{Outline of the paper}
Section \ref{sec:preliminaries} is devoted to preliminary results that will be used throughout the proofs of the main Theorems. Of particular interest are the interpolation inequalities, namely Propositions \ref{pro:interp1} and \ref{pro:interp2}. We will also recall well-known results on parabolic equations. They will be used for proving both results, especially Theorem \ref{thm:struwe_sol} where a bootstrap argument is needed.

We will prove Theorem \ref{thm:struwe_sol} in Section \ref{sec:proof_thm1}, which is divided into successive steps. We first collect uniform a priori estimates, namely Proposition \ref{pro:enDecay} and Corollary \ref{cor:bootstrap}, which will yield tightness of a sequence of stopped processes. Corollary \ref{cor:bootstrap} also provides a bootstrap for the solution, which will play a key role in the proof that the existence time is uniform with respect to a compact set of initial data. These arguments will ensure convergence towards a weak martingale solution. Noticing that the solution obtained has enough regularity to apply a basic Gronwall estimate, we will then make use of the celebrated Gy\"ongy-Krylov argument to yield convergence towards a local strong solution. By reiterating the process, we will be able to construct the ``Struwe solution'' on the whole interval $[0,T]$.

Theorem \ref{thm:uniqueness} will be treated in section \ref{sec:proof_thm2}. Writing ``Helein's decomposition'' of the nonlinearity $u|\nabla u|^2$, we show that the gradient part is controlled by the energy bound. This was already remarked in \cite{freire1995uniqueness}, but the new insight here is that we can define a process $\uu$, whose singular part has been removed.
The equation on $\nu :=u-\uu$ can then be linearized around the latter ``renormalized map'', and the supermartingale property will be used at this stage to obtain more regularity on the singular part, yielding then $\nu \in L^4(W^{1,4})$, from which the uniqueness follows. We point out that our method has the advantage of not referring to the Struwe solution, and therefore the proofs of Theorem \ref{thm:struwe_sol} and \ref{thm:uniqueness} (up to slight modification of the conclusion) are independent.
\section{Preliminaries}
\label{sec:preliminaries}
\subsection{Interpolation inequalities}
The following multiplicative inequality corresponds to a particular case of \cite[II Thm.\ 2.2]{ladyzhenskaya1968linear}.
\begin{proposition}
\label{pro:interp1}
There exists a constant $\mu _0>0$, such that for every $f\in H^1$ with $\int_{\T} f=0$ , there holds
\begin{equation}\label{ineq:interp}
\int_\T|f|^4\d x\leq \mu _0\left(\int_\T|f|^2\d x\right)\left(\int_\T|\nabla f|^2\d x\right)\,.
\end{equation}
\end{proposition}
As a byproduct, the following Lemma is obtained in
\cite[Lemma 3.1]{struwe1985evolution}. It will play a central role in the proof of Theorem \ref{thm:struwe_sol}.
\begin{proposition}
\label{pro:interp2}
For any $T>0$, there exists a constant $\mu _1>0$, such that for all $v\in C(0,T;H^1)\cap L^2(0,T;H^2)$, for all $\varrho>0$:
\begin{multline}
\label{ineq:interp_loc}
\iint_{[0,T]\times\T}|\nabla v|^4\d y\,\d t\leq \mu _1\left(\sup_{t\in [0,T],\, x\in\T}\int_{B(x,\varrho)}|\nabla v(t)|^2\d y\right)
\\
\times\left(\iint_{[0,T]\times\T}|\nabla^2 v|^2\d y\,\d t+\iint_{[0,T]\times\T}\frac{|\nabla v|^2}{\varrho^2}\d y\,\d t\right)\,.
\end{multline}
\end{proposition}
\subsection{Parabolic estimates for deterministic PDEs}
We recall regularity results associated to the deterministic Cauchy problem with unknown $\varphi$:
\begin{align}\label{eq:parabolicGene}
\begin{cases}
\partial _t\varphi -\Delta \varphi  = f(t,x)\,, &\text{in}\enskip [0,T]\times \T\, ,\\
\varphi (0,\cdot )=0\, ,&\text{in}\enskip \T\,.
\end{cases}
\end{align}
The following parabolic estimates are well known.
\begin{proposition}\label{pro:parab}
\begin{enumerate}[label=(\roman*)]
 \item Let $p\in(1,\infty)$ and $\alpha >0.$ For $f\in C^\alpha (0,T;L^p),$  Problem \eqref{eq:parabolicGene} has a unique solution in $C^1(0,T;L^p)\cap C(0,T;W^{2,p}).$
 \item For $f\in L^2(0,T;H^{-1}),$ Problem \eqref{eq:parabolicGene} has a unique solution in $H^1_0(0,T;H^{-1})\cap L^2(0,T;H^1).$
 \item Let $1<p,q<\infty.$ For $f\in L^q(0,T;L^p),$ Problem \eqref{eq:parabolicGene} has a unique solution in $W^{1,q}_0(0,T;L^p)\cap L^q(0,T;W^{2,p}).$
\end{enumerate}
Moreover, every solution above depends continuously on $f$ within the corresponding Banach spaces.
\end{proposition}
\begin{proof}
The first statement can be found in \cite{DPZ}.
The second and third statements can be found respectively in
\cite{lions1968problemes} and \cite{grisvard1969equations}.
\end{proof}
\subsection{Parabolic estimates: stochastic case}
We recall existence, uniqueness and regularity for weak solutions
of the parabolic equation with multiplicative noise:
\begin{equation}\label{eq:parabolicGene2}
\begin{cases}
\d Z -\Delta Z \d t  = \Psi (t)\d \bar W \,, &\text{ in }\Omega\times(0,\tau)\times \T\, ,
\\
Z (\cdot ,0)=0\, ,&\text{ in }\Omega\times\T\,,
\end{cases}
\end{equation}
(It\^o sense)
where this time $\bar W(t)\equiv\sum_{k\in\N}B_k(t)e_k$ is a cylindrical Wiener process.
Under suitable hypotheses on $\Psi$ (see the proposition below)
a weak solution $Z$ of \eqref{eq:parabolicGene2} exists and is unique.
It is given by the stochastic convolution, namely:
\begin{equation}\label{integral_Z_c3}
Z(t)=\int_0^tS(t-r)\Psi (r)\d \bar W\,,\enskip t\in[0,T]\, .
\end{equation}
\begin{proposition}
\label{pro:parab_sto}
Let $\alpha \geq 0,$ $p\in[2,\infty)$, $q\in[1,\infty)$, and let $\Psi $ be a progressively measurable process in $L^q\big(\Omega;L^q(0,T;\gamma ( W^{\alpha ,p}))\big).$ 
\begin{enumerate}[label=(\roman*)]
\item
For $p>2$, for every $\delta \in[0,1-2/q)$ and $\lambda \in [0,1/2-1/q-\delta /2),$ we have $Z\in L^q(\Omega;C^\lambda (0,T;W^{\alpha +\delta ,p}))).$ Moreover
\[
\E\left[\|Z\|^q_{C^\lambda (0,T;W^{\alpha +\delta,p})}\right]\leq C\E\left[\|\Psi \|^q_{L^q(0,T;\gamma(W^{\alpha ,p})) }\right]\,.
\]
\item For any $p\geq 2,$ $\delta \in(0,1),$ we have $Z\in L^q(\Omega;L^q(0,T;\gamma (W^{\alpha ,q}))),$ and
\[
\E\left[ \|Z\|^q_{L^q(0,T;W^{\alpha +\delta ,p})}\right]\leq C\E\left[\|\Psi \|^q_{L^q(0,T;\gamma (W^{\alpha ,p}))}\right]\,.
\]
\end{enumerate}
\end{proposition}
\begin{proof}
The first statement, the proof of which relies on the factorization method, is a particular case of \cite[Corollary 3.5]{brzezniak1997stochastic}.

The second point can be found in \cite{krylov1996theory}.
\end{proof}

\subsection{Local solvability}\label{sec:local_solv}
We first need to investigate local solvability of the It\^o equation
\begin{equation}
 \label{equation:v}
\d v=(\Delta v+v|\nabla v|^2+F_\phi v)\d t+v\times \d W\,,\quad v(0)=v_0\,,
\end{equation} 
for regular data $v_0$ and $\phi .$
To this aim we will switch to the mild formulation
\begin{equation}
\label{mild}
v(t)-S(t)v_0=\int_0^tS(t-r)(v|\nabla v|^2+F_\phi v)\d r+\int_0^t S(t-r)\left[v\times \d W\right]\,,
\end{equation}
$t\in[0,\tau )$ ($\tau $ sufficiently small), where $(S(t))_{t\in[0,T]}$ denotes the Heat semigroup $e^{t\Delta }$.
A triplet $(v;0,\tau )$ such that \eqref{mild} holds up to a stopping time $\tau >0$ will be called a \emph{local mild solution}.

Whenever $(v;0,\tau )$ is a local mild solution, $X$ is a Banach space and $q\in[1,\infty],$
we shall write $v\in L^q(\Omega ; C([0,\tau );X))$ to indicate that the stopped process $v(\cdot \wedge \sigma )$ belongs to
$L^q(\Omega;C(0,T;X))$ for any stopping time $\sigma <\tau .$

\begin{theorem}
\label{thm:localSolv}
Let $p\in(2,\infty)$, $\alpha >2/p,$ $q>2/\alpha ,$ and let $\phi \in \gamma (W^{\alpha ,p}).$
Then, for every $v_0\in W^{1,p}$ with $|v_0|=1$ a.e., there exists a unique $v\in L^q(\Omega;C([0,\tau );W^{1,p}))$
such that:
\begin{enumerate}[label=(\roman*)]
 \item $(v;0,\tau )$ is a local mild solution of \eqref{equation:v};
 \item on $\{\tau <T\}$, we have $\limsup_{t\to\tau }|v(t)|_{W^{1,p}}=\infty.$
\end{enumerate}
In addition, if $v_0\in W^{3,p}$ and $\phi\in \gamma(W^{2+\alpha  ,p}),$ there is another stopping time
$0<\tilde \tau \leq\tau$ such that $v\in L^q(C([0,\tilde\tau );W^{3,p})).$
\end{theorem}
\begin{remark}
\label{rem:sphere}
The solution above fulfills the norm constraint \eqref{sphere_constraint}, provided $t\in[0,\tau )$ is such that
$|\nabla u(t)|_{L^\infty}<\infty.$ This can be shown by an application of It\^o Formula to the functional $F(u):=|1-|u|^2|_{L^2}^2,$ together with Gronwall (see \cite{hocquet2015landau} for details). Hence, the embedding $W^{2,p}\hookrightarrow W^{1,\infty}$ for $p>2$ shows that any $W^{2,p}$ mild solution takes values in the sphere.
\end{remark}

\begin{proof}
The proof is based on a fixed point argument.
Note that since there is no pathwise estimate for stochastic terms, one cannot proceed to a fixed point in $L^q(\Omega;B)$, where $B$ is some ball in $C(0,T;W^{1,p}).$ This leads us to truncate the nonlinearities by the use of a
cut-off function $\theta:\R^+\rightarrow [0,1]$, which has the following properties:
\begin{equation}
\label{nota:theta}
\theta\in C_c^\infty(0,2)\,,\quad \theta(x)=1,\text{ for all } 0\leq x\leq1\,.
\end{equation}
For $R>0$, and $x\in\R^+$, we denote by $\theta_R(x)=\theta(\frac{x}{R}).$
We first solve the fixed point problem
$u=\psi^R(u),$
where for a fixed $R>0$,
we define the map $\psi^R$ on $L^2(\Omega;C(0,T;W^{1,p}))$ by the formula:
\begin{multline}
\psi^R(v)(t)=S(t)v_0+\int_0^{t}S(t-r)\left[\theta_R(|v(r)|_{W^{1,p}})v(r)|\nabla v(r)|^2\right]\d r
\\
+\int_0^tS(t-r)[F_\phi v(r)]\d r+\int_0^tS(t-r)[v(r)\times \d W(r)]\,,
\end{multline}
for $t\in[0,T].$
We show that provided $T$ is sufficiently small,
depending on $R,$ $\cc$ and $v_0$, then:
\begin{enumerate}[label=(\roman*)]
 \item\label{point_1} for $q>2/\alpha ,$ $\psi^R$ maps the ball 
\begin{multline*}
 B_{1+|v_0|_{W^{1,p}}}:=
 \Bigg\{v\in L^q\big(\Omega ;C(0,T;W^{1,p})\big),
 \\
 \text{such that}\enskip 
 \E\left[\|v\|^q_{C(0,T;W^{1,p}))}\right]^{1/q}<1+|v_0|_{W^{1,p}}
 \Bigg\}
\end{multline*}
 into itself;
 \item\label{point_2} $\psi^R|_{B_{1+|v_0|_{W^{1,p}}}}$ is a contraction.
\end{enumerate}
Then, Picard Theorem yields existence and uniqueness of a local mild solution $(u_R;0,\tau _R)$.

The proofs of properties \ref{point_1} and \ref{point_2} are straightforward consequences of Proposition \ref{pro:parab_sto},
as well as the well-known hypercontractivity property for $S,$ that is for $\lambda ,\mu \geq 0,$ and $1\leq r,q\leq\infty,$
there is a constant $K(\lambda ,\mu ,r,q)>0$ such that
\begin{equation}
\label{hypercontractivity}
|S(r)|_{\mathscr L(W^{\lambda ,q},W^{\mu ,r})}\leq \frac{K}{r^{\frac{\mu -\lambda }{2} +\frac{1}{q}-\frac{1}{r}}}\,,\quad r>0\,,
\end{equation}
see e.g.\ \cite[p.\ 25]{rothe1984global}.
We will content ourselves to show the property \ref{point_1}, where the main difficulties
are due to the stochastic convolution $Z:=\int_0^\cdot S(\cdot -r)v\times\d W$
and the nonlinear term. Computations leading to \ref{point_2} can be found in \cite{hocquet2015landau}.

First,
let $\epsilon >0$ such that
\[
\min(\alpha,1) >\epsilon >\frac{2}{p}\,.
\]
For such $\epsilon ,$ the space $W^{\epsilon ,p}$ is an algebra (this is immediately seen by using the definition of the fractional Slobodeckij spaces, together with the embedding $W^{\epsilon ,p}\hookrightarrow L^\infty$).
Consequently, using Proposition \ref{pro:parab_sto} we obtain 
\begin{multline*}
\E \left[\|Z\|_{C(0,T;W^{1,p})}^q\right]
\leq C T^q\E\left[\|u\times\phi \|_{C(0,T;\gamma (W^{\epsilon ,p}))}^q\right]
\\
\leq C T^q|\phi |_{\gamma (W^{\epsilon ,p})}^q\E\left[\|u\|_{C(0,T;W^{\epsilon ,p})}^q\right]\,,
\end{multline*}
and therefore:
\[
\E \left[\|Z\|_{C(0,T;W^{1,p})}^q\right]^{1/q}
\leq C' T|\phi |_{\gamma (W^{\alpha ,p})}\E\left[\|u\|_{C(0,T;W^{1,p})}^q\right]^{1/q}\,,
\]
for another such constant $C'>0.$

Concerning the nonlinearity, for any $t\in[0,T]$
we have the pathwise bound:
\begin{multline}
\label{sharp_inequality}
|\int_0^t S(t-r)(\theta _R(|v|_{W^{1,p}})v|\nabla v|^2)\d r|_{W^{1,p}}
\\
\leq K\int_0^t\theta _R(|v|_{W^{1,p}})\frac{|v|\nabla v|^2|_{L^{p/2}}}{(t-r)^{1/2+1/p}}\d r
\\
\leq KT^{1/2-1/p}\sup_{0\leq r\leq t}\left(\theta _R(|v(r)|_{W^{1,p}})|v(r)|_{L^\infty}|\nabla v(r)|_{L^p}^2\right)
\\
\leq CT^{1/2-1/p}R^3\,,
\end{multline}
where we made use of \eqref{hypercontractivity}.

Because $\cc\in\gamma (W^{\alpha ,p})\subset \gamma (L^p),$ 
the term $|\int_0^tS(t-r)F_\phi v\d r|_{W^{1,p}}$
is bounded above by
$C\int_0^t(t-r)^{-1/2-1/p}|F_\phi |_{L^{p/2}}|v|_{L^\infty}\d r$
which in turn is smaller than
$CT^{1/2-1/p}|\phi|_{\gamma (W^{\alpha ,p})}|v|_{W^{1,p}}.$

Since on the other hand, $|S(t)v_0|_{W^{1,p}}\leq |v_0|_{W^{1,p}}$
adding all the above contributions yields the smallness condition on $T:$
\[
T\in(0,T_R)\quad \text{with}\quad 
T_R:= \min\left\{1,\left(\frac{1}{\tilde C[R^3 + 2|\phi |_{\gamma (W^{\alpha ,p})}(1+|v_0|_{W^{1,p}})]}\right)^{\frac{1}{\min\left(1,\frac{p-2}{2p}\right)}}\right\},
\]
where we define $\tilde C$ as the biggest constant appearing in the previous computations.

The second property \ref{point_2} is similar (we might need to take $T_R$ smaller if necessary).
and therefore the existence and uniqueness of a fixed point $v_R$ follows for any $R>0$ and $T$ as above.

Using a localization procedure (see also \cite[Theorem 4.1]{de2002effect}) we can build a maximal solution as follows:
for $m\geq 1$
we define the stopping times
\begin{equation}\label{nota:tau_m}
\tau_m=\inf\left\{t\in[0,T],\enskip |v_m(t)|_{W^{1,p}}\geq m\right\}\, ,
\end{equation}
and show that
the sequence $(\tau_m)$ is non-decreasing and that $v_{m+1}(t)=v_{m}(t)$ for $t\in[0,\tau_m]$ a.s.\ 
The local solution $(v;0,\tau \equiv\sup_{m\geq 1}\tau _m)$ is then defined in an obvious way.

Higher regularity is obtained by writing the mild equation on $\Delta v,$ and again using regularizing properties of $S.$
\end{proof}
\begin{remark}\label{rem:Hilbert}
Instead of working with the scale $(W^{\alpha ,p})_{\alpha \geq 0},$ we could have considered as well the Hilbert spaces $(H^{\alpha })_{\alpha \geq 0},$
which are somehow easier to manipulate. Indeed, it can be shown that a local theory in $C(H^\alpha )$ works provided $\alpha >3/2.$
However, the choice of the Banach space $W^{1,p}$ for $p>2$ has the merit to stress the importance of the ``critical space'' $H^1.$
\end{remark}
\section{Proof of Theorem \ref{thm:struwe_sol}}
\label{sec:proof_thm1}
\subsection{Step 1: a priori estimates}\label{subsec:main_estimates}
If $u$ is a solution of \eqref{LLG-Torus},
an important role is played by the ``tension'', defined $\P\otimes \d t\otimes \d x$-almost everywhere by:
\begin{equation}\label{nota:tension}
\tt:=\Delta u+u|\nabla u|^2\,.
\end{equation}
According to our definitions \ref{def:StrongSol} and \ref{def:martingale_solution}, this quantity belongs to $L^2(\Omega \times[0,T]\times\T).$ Moreover,
thanks to the identity $\Delta |u|^2/2\equiv0=\Delta u\cdot u+|\nabla u|^2,$
it fulfills the geometrical property that
\begin{equation}\label{geometry_tt}
\tt= \Delta u - (u\cdot \Delta u)u\,,\quad
\end{equation}
namely it is pointwisely equal to the orthogonal projection of the laplacian onto the plane $\langle u \rangle^{\perp}$, up to a $\P\otimes\d t\otimes\d x$-null set.

Consider arbitrary positive numbers $\varrho$ and $\epsilon_1$ (to be fixed later).
If $u$ is supported in $C(H^1)$, one defines a stopping time $\zeta (u,\varrho;\epsilon _1)\leq T $ as follows:
\begin{equation}\label{nota:zeta}
\zeta (u,\varrho;\epsilon _1):=\inf\left\{0\leq t<T,\enskip\sup\limits_{x\in \T}\int\nolimits_{B(x,\varrho)}|\nabla u(t,y)|^2\d y\geq\epsilon_1\right\}\,
\end{equation}
(with the convention that $\zeta (u,\varrho;\epsilon _1)=T$ if the above set is empty).
Note that this definition also makes sense when $T$ is a stopping time.

Additionally, we introduce the notation $E(t):=\frac12\n{\nabla u(t)}{L^2}^2.$
\begin{proposition}[a priori estimates]\label{pro:enDecay}
Fix $\phi \in \hs^{1}$, and assume that $(u;0,T)$ denotes a local strong solution of \eqref{LLG-Torus}, where for convenience we suppose that $T>0$ is deterministic. Then
\begin{align}\label{en:item_i}
E(t)-E(0) +\int_0^t|\tt|^2_{L^2}\d r=t|\nabla\phi|^2_{\hs}+\int_0^t\langle \nabla u,u\times \d\nabla W \rangle\,,
\\
\intertext{a.s.\ for $t\leq T $. Moreover, letting $q\geq 1$, we have:}
\label{en:item_ii}
\E\left[\sup_{0\leq t\leq T }E(t)^q +\left(\int_0^T|\tt|_{L^2}^{2}\d t\right)^q \right]\leq C\left(q , E(0),|\phi|_{\hs^1}\right)\,,
\intertext{and there exists an absolute constant $\epsilon _1^*>0$ such that for any $\epsilon_1\in(0,\epsilon _1^*)$ and $\varrho>0$:}
\label{en:item_iii}
\E\left[\left(\int_0^{\zeta (u,\varrho;\epsilon _1)}|\Delta u|_{L^2}^2\d t\right)^q \right]\leq  C\left(\varrho,q  ,E(0),\n{\phi }{\hs^{1}}\right)\,,
\end{align}
The constants above depend on the indicated quantities, but not on $u$.
\end{proposition}
\begin{remark}
\label{rem:optimal}
The ``optimal value'' of $\epsilon _1^*$ corresponds to the inverse of $\mu _1$, namely the constant in the interpolation inequality given in Proposition \ref{pro:interp2} (in particular, it is independent of $T$ and $\varrho$).
\end{remark}
\begin{proof}[Proof of Proposition \ref{pro:enDecay}.]
\item[\indent\textit{Proof of \eqref{en:item_i}.}]
The solution $u$ has enough regularity to apply the It\^o Formula given in \cite{DPZ} (letting $H:=H^1$ and $F(u):=|\nabla u|_{L^2}^2/2$), so that 
\begin{equation}\label{ito_E}
E(t)-E(0)=\int_0^t\ps{-\Delta u}{\Delta u+u|\nabla u|^2}{}\d r+\int_0^t\ps{\nabla u}{u\times\circ\nabla \d W}{}\,.
\end{equation}
Moreover by our assumption that $|\phi |_{\hs^{1}}<\infty,$ the Stratonovitch integral makes sense as
$\int_0^t\ps{\nabla u}{u\times\circ\d \nabla W }{}
=1/2\int_0^t\sum_{\ell \in\N}T_\ell (r)\d r +\int_0^t\langle \nabla u,u\times\d \nabla W\rangle.$
Use now an adapted basis $(e_{\hat\ell })_{\hat\ell\in\{1,2,3\}\times\N},$ 
built over an orthonormal system $(f_\ell )_{\ell \in\N}$ of $L^2(\T;\R)$ in the following way: for $\ell \in\N,$ we set $e_{1,\ell }=(f_\ell ,0,0),$ $e_{2,\ell }=(0,f_\ell ,0),$ and $e_{3,\ell }:=(0,0,f_\ell )$. Denoting by $\cc_{\hat\ell }:=\cc e_{\hat\ell },$ we have for each $\hat\ell \equiv (k,\ell )\in\{1,2,3\}\times\N:$
\[
\begin{aligned}T_{\hat\ell}
&=\langle\nabla u\times \cc _{\hat\ell} ,u\times\nabla \cc _{\hat\ell} \rangle
+|u\times\nabla \cc _{\hat\ell} |_{L^2}^2
+\langle\nabla u,(u\times\cc _{\hat\ell} )\times\nabla \cc _{\hat\ell} \rangle
\\
&=:A_{\hat\ell}+B_{\hat\ell}+C_{\hat\ell}
\,.
\end{aligned}
\]
By \eqref{isotropic} and $|u|=1,$ we have on the one hand:
\[
\sum_{1\leq k\leq 3}B_{(k,\ell) }=2|\nabla \phi f_\ell |^2_{L^2}\,.
\]
On the other hand, using coordinates, there holds
\begin{equation}\label{hs1}
\sum_{1\leq k\leq3}A_{(k,\ell) }=-\int_{\T}\sum_{k,j}\partial _ju\cdot (u\times \partial _j\cc e_{(k,\ell) })\times\cc e_{(k,\ell) }
=\sum_{j}\int_{\T}(\partial _ju\cdot u )\phi f_\ell \partial _j(\phi f_\ell) =0\,,
\end{equation}
where we have used $\partial _ju\cdot u=0,$ for $1\leq j\leq2.$
Similarly, we have $\sum_{k\leq3}C_{(k,\ell) }=0,$ whence the It\^o correction is given by:
\[
\frac12\sum_{\hat\ell\in\{1,2,3\}\times\N}\int_0^tT_{\hat\ell }(r)\d r=t|\nabla \phi|_{\hs}^2\,.
\]

By \eqref{geometry_tt}, we have also $\ps{-\Delta u}{\Delta u+u|\nabla u|^2}{\R^3}=\ps{-\Delta u}{\tt}{\R^3}=-\n{\tt}{\R^3}^2,$ hence \eqref{en:item_i} follows.
\item[\indent\textit{Proof of \eqref{en:item_ii}.}]
Denoting the martingale term in \eqref{en:item_i} by 
$X(t):=\int_0^t\ps{\nabla u}{u\times\d \nabla W }{}$,
Burkholder-Davies-Gundy inequality writes
for any $q\in[1,\infty)$:
\begin{equation}\label{BDG_X1}
\E\left[\sup\nolimits_{t\in[0,T ]}|X(t)|^{q}\right]
\leq C(q)\E\left[\left(\int_0^T |\cc ^*\div(u\times\nabla u)|_{L^2}^2\d t\right)^{q/2}\right]\, .
\end{equation}
Since $\hs(L^2,H^1)\subset\mathscr L(L^2,H^1)$, we observe that 
\begin{equation}\label{first_martingale}
\E\left[\int_0^T |\cc ^*\div(u\times\nabla u)|_{L^2}^2\d t\right]
\leq C\big(\n{\phi}{\hs^{1}}\big)\E\left[\int_0^T \n{u\times\nabla u}{L^2}^2\d t\right]\,,
\end{equation}
and therefore the r.h.s.\ in \eqref{BDG_X1} is bounded above by
$C(q,\n{\phi }{\hs^{1}})\E[(\int_0^T E(t)\d t)^{q/2}].$
Going back to \eqref{en:item_i}, taking the power $q\geq 1$, the supremum, and the expectation, there comes
\begin{equation}\label{pre-gronwall}
\E\left[\sup_{0\leq \sigma \leq t}E(\sigma)^q\right]
\leq C\big( E(0),q,|\phi |_{\hs^{1}},T\big)\left(1+\int_0^t\E\left[\sup_{0\leq \sigma\leq r}E(\sigma)^q\right]\d r\right)\,,
\end{equation}
for any $t\in[0,T].$
Hence, the claimed bound follows by Gronwall Lemma.
Reusing \eqref{en:item_i}, \eqref{BDG_X1}, and injecting the latter bound gives the estimate on $\E[\nn{\tt}{L^2(L^2)}^{2q}].$
\item[\indent\textit{Proof of \eqref{en:item_iii}.}]
Let $\varrho,\epsilon _1>0,$ and $\zeta:=\zeta (u,\varrho;\epsilon_1).$
For $t\in[0,T ],$ using $\tt\perp u$ yields a.s.
\[\int_0^t|\tt|_{L^2}^2\d r=
\int_0^t\langle{\Delta u,\Delta u+u|\nabla u|^2}\rangle \d r\, .\]
Expanding this term in \eqref{en:item_i}, and still denoting by $X(t):=\int_0^t\langle{\nabla u,u\times \d\nabla W}\rangle$, we have in particular for $t=\zeta $
\[
\begin{aligned}
E(\zeta) -E(0)+\int_0^\zeta |\Delta u|_{L^2}^2\d t-X(\zeta )-\zeta |\nabla\phi |_{\hs}^2
&=\int_0^\zeta \langle{-\Delta u,u|\nabla u|^2}\rangle \d t
\\
 &\leq \frac12\int_0^\zeta |\Delta u|_{L^2}^2\d t+\frac12\int_0^\zeta |\nabla u|_{L^4}^4\d t\,,
\end{aligned}
\]
(since $|u|=1$ a.e.).
Using now Proposition \ref{pro:interp2} it follows by definition of $\zeta $ that
\[
\int_0^{\zeta }\n{\nabla u}{L^4}^4\d t\leq \mu _1\epsilon _1\int_0^{\zeta }\left(\frac{C}{\varrho^2}\n{\nabla u}{L^2}^2+\n{\Delta u}{L^2}^2\right)\d t\,,
\]
and this finally yields the estimate:
\begin{equation}\label{ineq:proof1}
\frac12(1-\mu _1\epsilon _1)\int_0^{\zeta }|\Delta u|_{L^2}^2\d t\leq \frac{C}{\varrho^2}\int_0^T |\nabla u|^2_{L^2}\d t+E(0)+\sup_{t\leq T}X(t)+T |\nabla\phi|_{\hs}^2\,.
\end{equation}
Provided $\epsilon _1$ is chosen to be sufficiently small, namely
$<\mu _1^{-1}$, then \eqref{en:item_iii} follows.
\end{proof}
\begin{remark}
Fixing $\epsilon _1<\epsilon _1^*,$ the above result can be improved to yield the exponential bounds:
\begin{align}
\label{exp_bound_1}
\E\left[\exp\left(m\sup_{0\leq t\leq\zeta }|\nabla u(t)|_{L^2}^2\right)\right]
&\leq C\left(m,\varrho\right)
\\
\label{exp_bound_2}
\E\left[\exp\left(m\int_0^\zeta  |\Delta u|^2_{L^2}\d t\right)\right]
&\leq C\left(m,\varrho,T,E(0),|\phi|_{\hs^{1}} \right)
\\
\label{exp_bound_3}
\E\left[ \exp\left(m\int_0^\zeta  |\nabla u|^4_{L^4}\d t\right)\right]
&\leq C\left(m,\varrho,T,E(0),|\phi|_{\hs^{1}} \right)\,.
\end{align}
for all $m\geq 1$.

\begin{proof}
The first bound is obtained as a consequence of the definition of 
$\zeta \equiv\zeta (u,\varrho;\epsilon _1),$
namely: write that for each $t\in[0,T]$, a.s.\ 
\begin{equation}\label{bound_rho_nabla}
\sup_{0\leq t\leq \zeta }E(t)
\leq \sum_{1\leq k\leq N_\varrho}\frac12\sup_{0\leq t\leq \zeta }\int_{B(x_k,\varrho)}|\nabla u(t,y)|^2\d y\leq N_\varrho \frac{\epsilon _1}{2}\,.
\end{equation}
where $\{x_1,\dots ,x_{N_\varrho}\}\subset \T$ denotes a finite set
such that 
$\bigcup_{1\leq k\leq N_\varrho}B(x_k,\varrho)=\T.$
Hence, $\exp(\sup_{t\leq T}mE(t))$ is integrable, and its norm is bounded by a constant depending on $m,\varrho,\epsilon _1.$
This proves \eqref{exp_bound_1}.

For the second bound, 
starting from \eqref{en:item_i}, and writing for simplicity
$X(t):=\int_0^{t\wedge \zeta }\langle \nabla u ,u\times\d\nabla W\rangle,$ 
computations similar to those leading to \eqref{ineq:proof1} show that
\[
m\int_0^t|\Delta u|^2_{L^2}\d r\leq C(1+ X(t))\,,\quad 
\text{a.s.\ for}\enskip t\in[0,\zeta ],
\]
where the constant above depends on the quantities $m,\varrho,E(0),T,|\phi |_{\hs^{1}},$
so that by the inequality $e^{a+b}\leq e^{2a}+e^{2b},$ it holds:
\begin{equation}\label{holds}
\E \left[\exp\left(m\int_0^\zeta |\Delta u|^2_{L^2}\d r\right)\right]
\leq \exp(2C)+ \E \left[\sup_{0\leq t\leq T} \exp(2CX(t))\right]\,,
\end{equation}
provided one can prove that the r.h.s. above is finite.
However, It\^o Formula on $Y(t):=\exp(mX(t))$ yields for every $t\leq \zeta :$
\begin{multline}
\label{pregron_Y}
Y(t)-Y(0)=m\int_0^tY(r)\langle\nabla u, u\times\d \nabla W\rangle + m^2\int_0^tY(r)|\cc ^*\div(u\times\nabla u)|_{L^2}^2\d r
\\
\leq m\int_0^tY(r)\langle \nabla u,u\times\nabla \d W\rangle + C(|\phi |_{\hs^{1}},\varrho,\epsilon _1,m)\int_0^tY(r)\d r\,,
\end{multline}
by \eqref{bound_rho_nabla}.
Taking the expectation in \eqref{pregron_Y} and applying Gronwall, we end up with the bound
\[
\sup_{0\leq t\leq T}\E\left[\exp(mX(t))\right]\leq C\left(m,\varrho,T,E(0),|\phi |_{\hs^{1}}\right)\,. 
\]
Now, from Doob's Inequality for submartingales, we also have 
\[\E\left[\sup_{0\leq t\leq T}\exp(CX(t))\right]\leq C'\sup_{0\leq t\leq T}\E \left[\exp X(t)\right]\,,\]
for another such constant.
This, together with \eqref{holds}, yields our second estimate \eqref{exp_bound_2}.

The bound \eqref{exp_bound_3} follows by combining \eqref{exp_bound_2}
with Proposition \ref{pro:interp1} and the estimate \eqref{bound_rho_nabla}.
\end{proof}
\end{remark}
\begin{corollary}[bootstrap]
\label{cor:bootstrap}
Consider $\phi \in\hs^3,$ $\varrho>0,$
fix $\epsilon _1>0$ as in Remark \ref{rem:optimal},
and let $(u;0,T)$ be a local strong solution of \eqref{LLG-Torus},
where again we assume for simplicity that $T>0$ is deterministic.
Assume in addition that $u$ is supported in $C(0,T;H^2)\cap L^2(0,T;H^3).$
Define the stopping time $\zeta \equiv\zeta (u,\varrho;\epsilon _1)$ as in \eqref{nota:zeta}.
Then, for every $m\geq 1,$ it holds:
\begin{equation}\label{add:estimate}
\E\left[\sup_{0\leq t\leq\zeta }|\Delta u(t)|^{2m}_{L^2}\right]\leq C\left(m,\varrho,T,|u_0|_{H^2},|\phi |_{\hs^3}\right)\,.
\end{equation}
Moreover the following bootstrap principle holds:
suppose
\[
\cc \in \bigcap_{k\in\N}\hs(H^k)\,,
\]
and assume in addition that the above solution verifies
$u(0)\in C^\infty(\T;\S).$
Then for each $k\in\N,$
we have 
\begin{equation}\label{bootstrap}
\zeta <\tau _{H^k}(u)\,,
\end{equation}
where $\tau _{H^k}(u)$ denotes the maximal existence time in $C(H^k)$. 
\end{corollary}
To prove this corollary, we need a refined version of Gronwall Lemma.
In what follows, by a \emph{superadditive function} on the simplex $ \Delta _T:=\{(s,t)\in[0,T]^2:s<t\}$, we mean a continuous map $\varphi:\Delta _T\to\R$ which equals zero on the diagonal and such that
\[
\varphi(s,\theta )+\varphi(\theta ,t)\leq \varphi(s,t)\,,
\]
for each $0\leq s\leq \theta \leq t\leq T.$
Moreover, a \emph{control function} is a superadditive map that is positive.

The following lemma is proved in \cite{deya2016priori}. Note here that we allow for a $\varphi_2$ which has no sign, however it is straightforward to check that the proof remains identical in this case.
\begin{lemma}\label{lem:rough_gronwall}
Fix $T>0$ and consider $G:[0,T]\to[0,\infty)$, continuous.
Let $\varphi_1:\Delta _T\to[0,\infty)$ denote a control function, and $\varphi_2:\Delta _T\to\R$ be superadditive. Assume that there exists $\kappa >0$ such that
for each $t,s\in[0,T]:$
\[
G_t-G_s\leq\left(\sup_{r\in[s,t]}G_r\right)\varphi_1(s,t)^{1/\kappa } +\varphi_2(s,t)\,.
\]
Then, there exists a constant $C_\kappa >0$ such that:
\begin{equation}\label{conclusion_RG}
\sup_{0\leq t\leq T}G_t\leq 2\exp\big(\max\{1,C_\kappa \varphi_1(0,T)\}\big)\left[G_0 +\sup_{0\leq t\leq T}|\varphi_2(0,t)|\right]\,.
\end{equation}
\end{lemma}
\begin{proof}[Proof of Corollary \ref{cor:bootstrap}]
\textit{Step 1: stochastic estimates.}
To prove the bound \eqref{add:estimate},
we first apply It\^o Formula to $1/2|\tt|^2_{L^2}$.
There comes 
\begin{multline}\label{ito_tt}
\frac12|\tt(t)|^2_{L^2}-\frac12|\tt(s)|_{L^2}^2 + \int_s^t|\nabla \tt(r)|^2_{L^2}\d r
\\
=\int_s^t\langle\tt,\tt|\nabla u|^2 + u\nabla u\cdot \nabla \tt \rangle\d r + M(s,t)
\end{multline}
for all $0\leq s\leq t\leq \zeta ,$ a.s.\ ,
where $M(s,t)\equiv M(t)-M(s)$ denotes the increment of the semi-martingale
\begin{equation}\label{semi_martingale}
M(t):=\int_0^t\langle \tt,\Delta (u\times\circ\d W) +|\nabla u|^2u\times\circ \d W+u\nabla (u\times\circ\d W)\cdot \nabla u \rangle\,.
\end{equation}
To estimate $\sup_{t\leq T} M(t)$, first expand the term $\Delta (u\times \circ\d W)$ so that the latter semimartingale rewrites as
\begin{multline}\label{decomp_M}
M(t)=\int_0^t\langle\tt, \tt \times\circ\d W \rangle
+\int_0^t\langle\tt, u\times \Delta \circ\d W \rangle
\\
+\int_0^t\langle\tt, (\nabla u\wedge) \circ \nabla \d W \rangle
+\int_0^t\langle\tt, u\left(\nabla (u\times\circ\d W)\cdot \nabla u\right)\rangle\,,
\\
=:A_t+B_t+C_t+D_t\,,
\end{multline}
where for a tensor $(f^\ell _j)_{\ell \leq3,j\leq2}$ we denote by 
\[
(\nabla u\wedge f)^\ell :=(\sum_{j}\partial _ju^{\ell +1}f^{\ell +2}_j-\partial _ju^{\ell +2}f^{\ell +1}_j)
\]
(here the index $\ell $ runs over $\mathbb{Z}/3\mathbb{Z}$).
We now evaluate each term of \eqref{decomp_M} separately.
Noting that $\tt\perp\tt \times W,$
it is clear that 
\begin{equation}\label{bound_At}
A_t=0\,.
\end{equation}
Similarly, 
by the fact that $u\perp\tt,$ we have for the last term:
\[
D_t=0\,.
\]
Concerning the second and the third terms, 
it is more convenient to use coordinates, write for instance 
\[
\begin{aligned}
C_t\equiv\sum_{\ell ,j}\int_0^t\langle \partial _ju^{\ell +1}\circ\d(\partial_j W^{\ell +2})-\partial _ju^{\ell +2}\circ\d(\partial_j W^{\ell +1}),\tt^\ell  \rangle
\\
=-\sum_{\ell ,j}\int_0^t\Big\{\langle u^{\ell +1}\circ\d(\partial_j W^{\ell +2})-u^{\ell +2}\circ\d(\partial_j W^{\ell +1}),\partial _j\tt^\ell \rangle
\\
+\langle u^{\ell +1}\partial _j\circ\d(\partial_j W^{\ell +2})-u^{\ell +2}\partial _j\circ\d(\partial_j W^{\ell +1}),\tt^\ell \rangle
\Big\}\,,
\end{aligned}
\]
so that $B_t+C_t=\int_0^t\langle\nabla \tt,u\times\nabla\circ\d W\rangle.$
Using the It\^o form of the latter Stratonovitch integral, 
we have
\begin{equation}\label{decomp:M}
M(t):=\hat M(t)+\int_0^t\sum_{\ell \in\N}\frac{1}{2}T_\ell (r)\d r\,,
\end{equation}
where $\hat M(t)$ is the corresponding It\^o integral
and can be estimated as follows,
using Burkholder-Davies-Gundy inequality:
\[
\E\left[\sup_{0\leq r\leq t}|\hat M(r)|^m\right]\leq C(m)\E\left[\left(\int_0^t|\cc ^*\div(u\times\nabla \tt)|_{L^2}^2\d r\right)^{m/2}\right]\,,
\]
for any $t\in[0,T]$ and $m\geq 1.$
Appealing to a similar argument as for \eqref{first_martingale}, we end up with
\begin{equation}\label{martingale_ineq_M}
\E\left[\sup_{0\leq r\leq t}|\hat M(r)|^m\right]
\leq C(m,|\phi |_{\hs^{1}})\E\left[\left(\int_0^t|\nabla \tt|_{L^2}^2\d r\right)^{m/2}\right]\,.
\end{equation}
It remains to estimate the trace term in \eqref{decomp:M}.
We have
\begin{multline*}
T_\ell :=
\langle\div(u\times\nabla  \cc _\ell ),\tt\times\cc _\ell \rangle
+\langle\div(u\times\nabla  \cc _\ell ),2\nabla u\wedge \nabla \cc _\ell \rangle
\\
+ \langle\div(u\times\nabla  \cc _\ell ),u\times\Delta \cc _\ell \rangle
+\langle\div(u\times\nabla  \cc _\ell ),2u(u\times\nabla \cc _\ell\cdot \nabla u ) \rangle
\\
+\langle\nabla \tt,(u\times\cc _\ell )\times\nabla \cc _ \ell  \rangle
=:\textstyle\sum\nolimits_{k=1}^5T^k_\ell 
\end{multline*}
where we have denoted by $\cc _\ell :=\cc e_\ell ,$
and also by $\langle\nabla \tt,(u\times\cc _\ell )\times\nabla \cc _ \ell  \rangle=\sum_{j=1,2}\langle\partial _j\tt,(u\times\cc _\ell )\times\partial _j\cc _ \ell  \rangle.$
Straighforward but cumbersome computations show that we have a bound
\begin{equation}\label{bound:trace_M}
\E\left[\left(\int_0^T\sum_{\ell \in\N}\frac{1}{2}T_\ell (r)\d r\right)^m\right]
\leq C\left(|\phi |_{\hs^3},m,\varrho,T,E(0)\right)\,.
\end{equation}
For instance:
\[
T^1_\ell \leq |\nabla u|_{L^4}|\nabla \cc_\ell |_{L^4}|\tt|_{L^2}|\cc _\ell |_{L^\infty}
\leq C|\cc _\ell |_{H^2}^2 (|\nabla u|^2_{L^2}+|\Delta u|^2_{L^2}
+ |\tt|_{L^2}^2)\, ,
\]
using again the interpolation inequalities.
Again, we have
\[T^2_\ell \leq C|\nabla \cc _\ell |_{L^2}(|\Delta u|_{L^2}|\nabla \cc_\ell  |_{L^\infty}+|\nabla u|_{L^4}|\Delta \cc _\ell |_{L^4})
\leq C|\cc_\ell |_{H^1}(|\Delta u|_{L^2}^2 + |\nabla u|^2_{L^2}+ |\cc_\ell |_{H^3}^2)\,,\]
and the remaining terms are estimated in the same way.
Summing over $\ell \in\N$, integrating in time and using the energy estimates, we end up with \eqref{bound:trace_M}.

\item[\indent\textit{Step 2: bound on the Laplacian.}]
From \eqref{ito_tt} and $\tt\perp u$ we have for any $0\leq s\leq t\leq \zeta ,$ a.s.:
\[\frac12\Big(|\tt(t)|^2_{L^2}-|\tt(s)|_{L^2}^2\Big) + \int_s^t|\nabla \tt|_{L^2}^2\d r
=\int_s^t \langle |\tt|^2, |\nabla u|^2\rangle\d r+M(s,t)\,.
\]
Using H\"older and \eqref{ineq:interp}, there comes:
\begin{multline}
\label{ineq_tt}
\frac12\Big(|\tt(t)|_{L^2}^2
-|\tt(s)|_{L^2}^2\Big)+\int_s^t|\nabla \tt|_{L^2}^2\d r
\\
\leq\left(\int_s^t|\nabla u|^4_{L^4}\d r\int_s^t|\tt|_{L^4}^4\d r\right)^{1/2} +M(s,t)
\\
\leq\frac{\sqrt{\mu _0}}{2}\left(\int_s^t|\nabla u|^4_{L^4}\d r\right)^{1/2}\left(\sup_{r\in[s,t]}|\tt(r)|_{L^2}^2+\int_s^t|\nabla \tt|^2_{L^2}\right)+M(s,t)\,.
\end{multline}

Hence, we can apply Lemma \ref{lem:rough_gronwall} with
$\varphi_1 (s,t):=\mu _0/4\int_s^t|\nabla u(r)|^4_{L^4}\d r,$
$\varphi_2 (s,t):=M(s,t),$
$\kappa =2$
and
$G_t:=|\tt(t)|_{L^2}^2+\int_0^t|\nabla \tt|^2_{L^2}\d r.$
This yields the pathwise estimate
\begin{equation}\label{gronwall_struwe}
\sup_{t\in[0,\zeta ]}G_t
\leq C\left(1+\exp\int_0^T|\nabla u|^4_{L^4}\d r\right)
\left(|\tt(0)|_{L^2}^2+ \sup_{t\in[0,\zeta ]} |M(t)|\right)\,,
\end{equation}
for some universal constant $C>0.$
Using now the exponential bound \eqref{exp_bound_3}, \eqref{martingale_ineq_M} and \eqref{bound:trace_M},
there exists a constant $C$ depending on the quantities $|\tt(0)|,\varrho,|\phi |_{\hs^{3}},T,\epsilon _1,$ such that
\begin{multline}
\label{bound_tt_1}
\E\left[\sup_{t\in[0,\zeta ]}|\tt(t)|_{L^2}^2+\int_0^\zeta |\nabla \tt(r)|^2_{L^2}\d r\right]
\leq C\Big(1 +\E\left[\exp2\int_0^\zeta |\nabla u|^4_{L^4}\d r\right]^{1/2}
\\
\times\E\left[\int_0^\zeta |\nabla\tt|^2_{L^2}\right]^{1/2}\Big) \,,
\end{multline}
which is bounded above by
\[
C\left(1+\frac \delta 2\E\left[\exp2\int_0^\zeta |\nabla u|^4_{L^4}\d r\right] +\frac{1}{2\delta }\E\left[\int_0^\zeta |\nabla \tt|^2_{L^2}\right]\right)\,,
\]
for any $\delta >0.$
Choosing $\delta $ smaller than $C$ and then
absorbing to the left in \eqref{bound_tt_1}, we obtain
\begin{equation}\label{bound_tt_2}
\E\left[\sup_{t\in[0,\zeta ]}|\tt(t)|_{L^2}^2+\int_0^\zeta |\nabla \tt(r)|^2_{L^2}\d r\right]\leq C\left(|\tt(0)|,\varrho,|\phi |_{\hs^{3}},T,\epsilon _1\right) \,.
\end{equation}

Now, using $\tt\perp u,$ and then applying Proposition \ref{pro:interp2} to the constant function $v\equiv u(t,\cdot ),$ we have for all $t\in[0,\zeta ]:$
\begin{multline}
|\Delta u(t)|_{L^2}^2\equiv|\tt(t) -u(t)|\nabla u(t)|^2|_{L^2}^2
=|\tt(t)|_{L^2}^2+|\nabla u(t)|_{L^4}^4
\\
\leq |\tt(t)|_{L^2}^2+\mu _1\epsilon _1\left(|\Delta u(t)|^2_{L^2}+\frac{C}{\varrho^2}|\nabla u(t)|_{L^2}^2\right) \,.
\end{multline}
Since $\epsilon _1<(\mu _1)^{-1},$
we end up with \eqref{add:estimate}.

\item[\indent\textit{Step 3: increasing the regularity of the stochastic convolution.}]
We appeal here to the same arguments as that of \cite{debussche2015regularity}:
define the stochastic convolution:
\begin{equation}\label{decomp:uyz}
Z(t):=\int_0^tS(t-r)u\times\d W\,,\quad t\in[0,\zeta ]\,,
\end{equation}
and for simplicity denote by $L^m(L^2):=L^m(0,\zeta ;L^2),$ $C(L^2):=C(0,\zeta ;L^2)$ and so on.
Using \eqref{add:estimate}, Proposition \ref{pro:parab_sto} yields that for every 
$4<m<\infty$, fixing for instance $\delta :=1/2<1-2/m$ we have with $\lambda =0:$
\begin{equation}\label{estimate:z}
\E\left[\|Z\|^{m}_{C(H^{5/2})}\right]\leq C\E\left[\|u\times \phi \|_{L^m(\hs^{2})}^m\right]\leq C'\left(|\phi |_{\hs^{4}},T\right)\left(1+\E\left[\|\Delta u\|^m_{C(L^2)}\right]\right)\,,
\end{equation}
where the second inequality comes from 
the fact that for any $k\in\N:$
\[
|u\times \phi e_k|_{H^2}\leq |u|_{H^2}|\phi e_k|_{W^{2,\infty}}\,,
\]
together with the embedding $H^4\hookrightarrow W^{2,\infty}.$

\item[\indent\textit{Step 4: increasing the regularity of the solution.}]
Observe that $y:= u-Z$ is a solution of the following PDE with random coefficents:
\begin{equation}\label{random_PDE}
\partial _ty-\Delta y=u|\nabla u|^2\,.
\end{equation}
However, from \eqref{add:estimate} and the Sobolev embedding theorem in dimension two, we can deduce that:
\begin{equation}\label{reg:f}
f\equiv u|\nabla u|^2 \in C(L^p)\,,\quad \text{for any}\enskip p\in[1,\infty)\,,\enskip \text{a.s.\ ,}
\end{equation}
so that using Proposition \ref{pro:parab}, we have in particular $y\in C(W^{2,4})$ and
\begin{equation}\label{bound:y}
\E\left[\|y\|_{C(W^{2,4})}^m\right]\leq C_p\left(1+\E\left[\|\nabla u\|_{C(L^8)}^{2m}\right]\right)\leq C_p\left(1+\E\left[\|\Delta u\|^{2m}_{C(L^2)}\right]\right)\,
\end{equation}
for every $m\geq 1.$
Observe that by \eqref{estimate:z} and $H^{5/2}\hookrightarrow W^{2,4},$ there holds:
$\E[\|Z\|_{C(W^{2,4})}^m]\leq C(1+\E[\|\Delta u\|_{L^m(L^2)}^m]),$
and a similar bound holds for $u=y+Z,$ namely
\begin{equation}\label{estimate:u2}
\E\left[\|u\|_{C(W^{2,4})}^m\right]\leq C(|\phi |_{\hs^{4}}) \left(1+\E\left[\|\Delta u\|_{L^m(L^2)}^m\right]\right)\,.
\end{equation}
We have now 
$\nabla f\equiv \nabla u|\nabla u|^2+2u\nabla ^2u\nabla u\in C(L^2)$
whence
$f\in C(H^1)$
and
\begin{multline}\label{finally}
\E\left[ \|y\|_{C(H^3)}^m\right]
\leq C \E\left[\|u|\nabla u|^2\|^m_{C(H^1)}\right]
\\
\leq C \E\left[\left(1+\|\nabla u\|_{C(L^4)}^2\|\nabla u\|_{C(L^2)}+2\|\nabla^2u\|_{C(L^4)}\|\nabla u\|_{C(L^4)}\right)^m\right]
\\
\leq \E\left[ P_m(\|\Delta u\|_{C(L^2)})\right]\,,
\end{multline}
where $P_m$ is a polynomial.

We can now repeat Step 3 to obtain
\begin{equation}\label{estimate:z2}
\E\left[\|Z\|^m_{C(H^3)}\right]\leq \E \left[P(\|\Delta u\|_{C(L^2)})\right] ,
\end{equation}
and finally 
\begin{equation}\label{estimate:z2}
\E \left[\|u\|_{C(H^3)}^m\right]
\leq \E \left[P(\|\Delta u\|_{C(L^2)})\right]<\infty\,,
\end{equation}
for another such polynomial, depending on $m,|\phi |_{\hs^{5}}.$
Reiterating the argument above, a straightforward induction shows that
provided $\phi \in\cap_{k\in\N}\hs(H^k),$
then
\[
u(\cdot \wedge \zeta) \in \bigcap_{k\in \N} L^m(\Omega; C(H^k))\,.
\]
This finishes the proof of Corollary \ref{cor:bootstrap}.
\end{proof}
\begin{remark}
The reason why a bootstrap argument is needed will be seen in \eqref{strict} and \eqref{P_zero_boot}.
Whenever
\[
t=\zeta (u,\varrho;\epsilon _1)
\]
for some $\varrho>0,$ where $\epsilon _1>0$ is taken smaller than $\epsilon _1^*$ (see Remark \ref{rem:optimal}), the bootstrap ensures the possibility to extend the solution during a positive time after $t$, in a space where It\^o formulas \eqref{ito_E} and \eqref{ito_tt} are licit.

However the regularity ``$u\in C(H^3)$'' turns out to be sufficient to make them rigorous, and hence to prove Theorem \ref{thm:struwe_sol}.
\end{remark}
\subsection{Step 2: Tightness}
We now define a sequence $\{W_n,n\in\N\}$ of Wiener processes in $L^2(\T;\R^3)$
where for each $n\in\N$,
$W_n$ is given by the sum
\begin{equation}\label{nota:Wn}
W_n:=\sum_{\ell \in\N} B_\ell (\cdot )\cc_ne_\ell \, ,
\end{equation}
for $(e_\ell ),(B_\ell )$ as above, and $\cc _n\equiv(\phi _n\cdot ,\phi _n\cdot ,\phi _n\cdot )$ denotes a sequence of Hilbert-Schmidt operators.
Consider the regularized problem:
\begin{equation}\label{approx}
\d v_n=(\Delta v_n+v_n|\nabla v_n|^2+F_{\phi _n}v_n)\d t+v_n\times \d W_n\,.
\end{equation}
We make the following assumptions:
\begin{enumerate}[label=(A\arabic*)]
 \item\label{Hyp1}
For all $n\in\N,$ we have $v_n(0)\in C^\infty(\T;\S)$, moreover: $v_n(0)\rightarrow u_0$ in $H^1$ and $\frac12|\nabla v_n(0)|_{L^2}^2\leq CE(0)\equiv\frac{C}{2}|\nabla u_0|_{L^2}^2$;
\item\label{Hyp2}
For all $n\in\N$, we have $\phi _n\in\cap_{k\in\N}\hs^k$, and $\phi _n\rightarrow \phi$ in $\hs(L^2;H^1)$;
\end{enumerate}
Note that \ref{Hyp2} is possible by considering e.g.\ the sequence of finite range operators $\cc _n:=\sum_{k\leq n}\cc e_k \langle e_k,\cdot  \rangle.$
Furthermore thanks to Theorem \ref{thm:localSolv}
(and also \cite[Prop.\ 6.4 p.\ 162]{DPZ}),
the assumptions \ref{Hyp1} and \ref{Hyp2} ensure that
\begin{enumerate}[label=(A\arabic*)]
\setcounter{enumi}{2}
\item\label{Hyp3}
For every $n\in\N$, there exists a unique maximal strong solution $(v_n;0,\tau _n)$ to \eqref{approx}, having continuous paths with values in $H^3.$ We have the property
\[
\tau _n=T\quad\text{or}\quad\limsup_{t\to\tau}|v_n(t)|_{H^3}=\infty\,.
\]
\end{enumerate}
Now, let $\epsilon _1^*:=\mu _1^{-1}$ where $\mu _1$ is as in \eqref{ineq:interp_loc},
fix $\epsilon_1\in(0,\epsilon _1^*)$, choose a non-increasing, positive sequence $\varrho_k\to0,$ $k\to\infty.$
For each $n,k\in\N,$ define the following stopping times:
\begin{equation}\label{zeta_tightness}
\zeta _{n,k}:=
\zeta (v_n,\varrho_k;\epsilon _1)
\equiv\inf\left\{0\leq t<\tau_n,\enskip\sup\limits_{x\in \T}\int\nolimits_{B(x,\varrho_k)}|\nabla v_n(t,y)|^2\d y\geq\epsilon_1\right\}\,,
\end{equation}
and denote by $u_{n,k}$, $k\in\N$, the ``mildly stopped process'':
\begin{equation}\label{nota:u_nk}
u_{n,k}(t)=
\begin{cases}
 v_n(t)\qquad\text{if}\enskip0\leq t\leq \zeta_{n,k}\,,\\
 e^{-(t-\zeta _{n,k})\Delta ^2}v_n(\zeta_{n,k})\qquad\text{if}\enskip\zeta_{n,k}<t\leq T\,,
\end{cases}
\end{equation}
(the reason for this definition will become clearer in the proof of the claim below).

We will also denote by 
\begin{equation}\label{nota:Unk}
U_n:=\{u_{n,k};k\in\N\}\,,\enskip n\in\N\, .
\end{equation}

\begin{claim}
For every $\delta<1$,
the sequence $\{U_n,n\in\N\}$ is tight in $E:=\prod_{k\in\N}L^2(0,T;H^{1+\delta })\cap C(0,T;H^{\delta })$.
\end{claim}

\begin{proof}
The proof is rather similar than that of \cite[Lemma 4.2]{brzezniak2013weak}.
It uses the a priori estimates, together with the following classical compactness result (``Aubin-Lions Lemma''):
{\itshape If $B_0\subset B\subset B_1$ are Banach spaces, such that $B_0,B_1$ are reflexive, and the embedding of $B_0$ in $B$ is compact, and if $(\beta,p,q)\in(0,1)\times (1,\infty)\times (1,\infty)$ with $\beta p>1$
then
$L^q(0,T;B_0)\cap W^{\beta,p}(0,T;B_1)\hookrightarrow L^q(0,T;B)$ and
$C(0,T;B_0)\cap W^{\beta,p}(0,T;B_1)\hookrightarrow C(0,T;B)$ (with compact embeddings).
}%

\item[\indent\textit{Regularity in time.}]
We need uniform estimates in some space $W^{\beta,p}(0,T;B_1)$,
where $B_1$ can be any reflexive Banach space containing $L^2$, and $\beta p>1$. These bounds essentially follow from the equation on $v_{n ,k}$.

As in the proof of Lemma 4.1 in \cite{brzezniak2013weak},
we can write, using \eqref{approx}:
\[\begin{aligned}
 v_n(t)-v_n(0)&=\int_0^t\left(\Delta{v_n}+v_n|\nabla{v_n}|^2\right)\d r+\int_0^tF_{\phi _n}v_n \d r +\int_0^tv_n\times dW_n(r)\\
 &= J_n ^1(t)+J_n ^2(t)+J_n ^3(t)\, ,
\end{aligned}\]
for all $n\in\mathbb{N}$ and $t\in[0,\tau _{n})$, a.s.
Recall that this equation holds in the sense of Bochner, and It\^o integrals in $L^2$.

Now, the bound
$\E[\|J_n ^1\|_{W^{1,2}(0,\tau _{n };L^2)}^2]\leq C(|\phi|_{\hs(H^1)})$
is a consequence of the uniform estimate \eqref{en:item_ii}, and
by the definition of the correcting term $F_{\phi_{n }}v_{n }$ and \ref{Hyp3}, we obtain
$\E[\|J_n ^2\|_{W^{1,2}(0,\tau _{n };L^2)}^2]\leq C(|\phi|_{\hs(L^2)})$.
Lastly, using Lemma 2.1 from \cite{flandoli1995martingale}, for any $\beta \in(0,\frac{1}{2})$, $\infty>p\geq2$ there exists a constant depending only on $\beta ,p,|\phi|_{\hs}$ such that:
$\E[\|J_n ^3\|_{W^{\beta ,p}(0,\tau _{n };L^2)}^p]\leq C(\beta ,p,|\phi|_{\hs})$.
Putting these bounds together, we have for each $n ,k\in\N$:
\begin{equation}\label{J}
\E\left[\nn{u_{n,k}}{W^{\alpha ,q}(0,T;L^2)}\right]
\leq C\left(\alpha ,\n{\phi }{\hs(H^1)}\right)\,,
\end{equation}
for some $1\geq \alpha>0$ and $q\geq 1$ with $\alpha q>1$, depending on $\beta ,p.$

\item[\indent\textit{Bounds on the whole time interval and conclusion.}]
Applying Proposition \ref{pro:enDecay}, we have for all $n,k\in\mathbb{N}$:
\begin{equation}\label{estimates_proof_tightness}
\E\left[\sup_{0\leq t\leq \zeta _{n ,k}}|\nabla{u_{n,k}}|_{L^2}^2+\int_0^{\zeta _{n ,k}}|\Delta u_{n,k}|_{L^2}^2\d t\right]\leq C\left(k, E(0),\n{\phi}{\hs^{0,1}}\right)\, .
\end{equation}
The fact that this uniform bound holds on the whole interval $[0,T]$ (and not only on $[0,\zeta_{n ,k} )$) is however not clear. This is precisely the reason why we extend $u_{n,k}$ after $\zeta _{n ,k}$ by the solution of a linear parabolic equation involving the bilaplacian, see \eqref{nota:u_nk}. This technical tool allows to ``forget'' the value of $\n{\Delta u(\zeta _{n ,k})}{L^2}$ (on which we have no control when $\cc _n$ is not bounded in $\hs^3,$ see \eqref{add:estimate}).
Indeed: for the sectorial operator $A:=\Delta ^2$, $D(A):=H^4$, we have the classical inequality
\[
|e^{-tA}f|_{D(A^{1/2})}\leq C\frac{|f|_{D(A^{1/4})}}{t^{1/4}}\, ,\enskip \text{for}\enskip t>0\, ,\enskip\text{and}\enskip  f\in H^1\, .
\]
Therefore, by the definition \eqref{nota:u_nk} we have
\begin{equation}\label{whole_interval}
\E\left[\int_{\zeta }^T|\Delta u_{n ,k}(t)|_{L^2}^2\d t\right]\leq C\E\int_{\zeta_{n ,k} }^T\frac{|\nabla u(\zeta_{n ,k} )|^2_{L^2}}{(t-\zeta_{n ,k} )^{1/2}}\d t\enskip ,
\end{equation}
which is bounded by a constant $C(E(0),T,\phi )$, using \eqref{en:item_ii}.

Using Proposition \ref{pro:enDecay},
we have for all $n,k\in\mathbb{N}$:
\begin{equation}\label{estimates_proof_tightness}
\E\left[\sup_{0\leq t\leq T}|\nabla{u_{n,k}}|_{L^2}^2+\int_0^T|\Delta u_{n,k}|_{L^2}^2\d t\right]
\leq C\left(k,E(0),T,\n{\phi}{\hs^{1}}\right)\, .
\end{equation}
The tightness follows: for $\delta <1$, set first $B_0=H^1,B=H^\delta ,=B_1=L^2$, and then $q=2$, $B_0=H^2,B=H^{1+\delta },B_1=L^2$,
so that the embedding
\[
C(0,T;H^1)\cap W^{\beta ,p}(0,T;L^2)\cap L^2(0,T;H^2)\hookrightarrow C(0,T;H^\delta )\cap L^2(0,T;H^{1+\delta })
\]
is compact by Aubin-Lions Lemma.
We conclude using the estimates \eqref{J}--\eqref{estimates_proof_tightness},
together with Tychonov Theorem, Markov inequality.
We refer the reader to \cite{hocquet2015landau} for details.
\end{proof}
By classical properties of Wiener processes the sequence $\{(U_n ,Z_{n },W_{n }),n\in\N\}$ is also tight in $E\times \prod_{k\in\N}[0,T]\times C^\alpha (0,T;H^1)$ for some $\alpha \in(0,\frac12)$, where we let for $n\in\N$:
\begin{equation}\label{nota:Zeta}
Z_n:=\{\zeta _{n,k},k\in\N\}\,.
\end{equation}
Therefore, by Prokhorov Theorem there exists an extraction $n_\ell ,\ell \in\N$, and a law $\mu $ supported in $\prod_{k\in\N}\left(L^2(0,T;H^2)\cap C(0,T;H^1)\right)\times \prod_{k\in\N}[0,T]\times C^\alpha (0,T;H^1)$ such that
$\law(U_{n_\ell },Z_{n_\ell},W_{n_\ell })\to \mu $ weakly.
By a standard application of Skorohod theorem, we however obtain a little more.
\begin{corollary}\label{cor:conv}
There exist
\begin{itemize}
 \item a stochastic basis $\PPs=(\Omegas,\Fs,\Ps,(\Fs_t)_{t\in[0,T]},\Ws)$, where $\Ws$ is a Wiener process in $L^2$ with covariance $\cc\cc^*$;
 \item a sequence of random variables $\left\{\left(\{\us_{\ell ,k}\}_{k\in\N},\{\zetas _{\ell ,k}\}_{k\in\N} ,\Ws_\ell \right) \,,\ell \in\N\right\}$,
where for each $\ell,k \in\N$, $\us_{\ell ,k}:\Omegas\to C(0,T;H^1)\cap L^2(0,T;H^2)$ denotes a predictable process, and $\zetas_{\ell ,k}$ is a positive stopping time, whereas $\Ws_\ell$ is a $\cc_{n_\ell }\cc_{n_\ell }^*$-Wiener process with respect to $(\Fs_t)$;
\item limits $\us_k(\omega ')\in\cap_{\delta <1}C(0,T;H^{\delta })\cap L^2(0,T;H^{1+\delta })$, and $\zetas_k(\omega ')\in[0,T] $, for every $k\in\N$,
\end{itemize}
such that the following convergences hold for each $k\in\N$:
\begin{align}\label{cv1}
&\us_{\ell ,k}\to\us_k\quad\Ps\text{-a.s.}
\\
&\qquad \text{in every}\enskip C(0,T;H^{\delta })\cap L^2(0,T;H^{1+\delta })\enskip \text{for}\enskip \delta <1;\nonumber
\\
\label{cv2}
&\zetas_{\ell ,k}\to\zetas _k\quad \Ps\text{-a.s.};
\\
\label{cv3}
&\Ws_\ell \to \Ws \quad \Ps\text{-a.s.\ in every}\enskip C^\alpha (0,T;H^1)\,,\enskip \text{for}\enskip \alpha <\frac12;
\\
\label{cv4}
&\textstyle\E\left[\langle \int_0^{\cdot }\us_{\ell ,k}\times\d \Ws_\ell ,X\rangle\right]
\to\E\left[\langle\int_0^{\cdot }\us_k\times\d \Ws,X\rangle\right]\,,
\\
&\qquad\text{for every predictable process $X$ in $L^2(\Omegas\times[0,T]\times\T).$}\nonumber
\end{align}
\end{corollary}
\begin{proof}
The proof is standard.
These properties are a consequence of Skorohod embedding Theorem (see \cite{watanabe1981stochastic}),
and classical properties of Wiener processes: we write that
\begin{equation}
\label{Id:martingale_1}
\Es\left[\left(M_{\ell ,k}(t)-M_{\ell ,k}(s)\right)\varphi\big(\us_{\ell,k}|_{[0,s)} ,\Ws_\ell|_{[0,s)}\big)\right]=0\,,\\
\end{equation}
\begin{multline}
\label{Id:martingale_2}
\Es\Big[\Big(\ps{M_{\ell ,k}(t)}{a}{}\ps{M_{\ell ,k}(t)}{b}{}
-\ps{M_{\ell ,k}(s)}{a}{}\ps{M_{\ell ,k}(s)}{b}{}
\\
- \int_s^t\ps{\us_{\ell ,k}\times\cc _{n_\ell}  a}{\us_{\ell,k} \times\cc _{n_\ell} b}{}\d r\Big)
\times\varphi\big(\us_{\ell,k}|_{[0,s)} ,\Ws_\ell|_{[0,s)}\big)\Big]=0\,,
\end{multline}
for any $\ell ,k\in\N$, $0\leq s\leq t\leq T$, $a,b\in L^2$ and $\varphi $ bounded continuous, where $M_{\ell ,k}(t)=\us_{\ell ,k}(t)-\us_{\ell ,k}(0)-\int_0^t(\Delta \us_{\ell ,k}+\us_{\ell ,k}|\nabla \us_{\ell ,k}|^2+F_{\phi _{n_\ell} } \us_{\ell ,k})\d r$.
We can then take the limits in \eqref{Id:martingale_1} and \eqref{Id:martingale_2}, and apply the Martingale Representation Theorem (see \cite{DPZ}).
Details of this argument can be found in the monograph \cite{skhorokhod2014studies}, see also \cite{alouges2014semi}.
\end{proof}
\subsection{Step 3: below estimates for $\lim_{n\to\infty}\zetas _{k,n}$}
Uniform bounds from below for the stopping time $\zetas _{n,k}$ will garantee the existence of the ``Struwe solution'' during a positive time, and therefore the present section can be considered, together with the justification of the bootstrap (namely Corollary \ref{cor:bootstrap}), as the core of the argument.
By strong convergence of $\us_{\ell ,k}(0)$ towards $\us_k(0)$ in $H^1$, and the fact that $\varrho_k\to0$,
we can assume without restriction that for some $\lambda \geq 2$, and for all $k\in\N$:
\begin{equation}\label{k_geq_k0}
\sup_{x\in\T}\int_{B(x,\lambda \varrho_k) }|\nabla \us_{\ell ,k}(0,y)|^2\d y\leq \frac{\epsilon _1}{2}\, ,\enskip \text{uniformly in}\enskip \ell \in\N\, ,
\end{equation}
(note that we also use compactness of $\T$ here).
We will always assume \eqref{k_geq_k0} in the following.
\begin{proposition}
\label{positivity}
For each $k\in\N$,
the limit point $\zetas _k$ of the sequence $\{\zetas_{\ell ,k} ,\ell \in\N\}$ (see Corollary \ref{cor:conv}) verifies
\[\Ps(\zetas _k>0)=1\, .\]
\end{proposition}
To prove the claim, we need the following local estimate.
\begin{lemma}
\label{LD}
Let $\eta \in C^\infty_0(\T;\R)$, $\varrho>0$, and $x\in\T$,
such that $\spt(\eta)\subset B(x,\varrho)$ and $|\nabla \eta|_{L^\infty}\leq \frac{C}{\varrho}.$
Then, for every local strong solution $(u;0,\tau ),$ there holds a.s.\ for $t\in[0,\tau )$:
\begin{equation}\label{formula-LocalEnergy}
\frac{|\eta \nabla u(t)|_{L^2}^2}{2}-\frac{|\eta \nabla u(0)|_{L^2}^2}{2}\leq t|\eta\nabla \phi |_{\hs}^2+\frac{C}{\varrho^2}\int_0^t|\nabla u|_{L^2}^2\d r
+\int_0^t\langle{\eta\nabla u,\eta u\times\nabla \d W }\rangle\,,
\end{equation}
where we denote by
$|\eta\nabla \phi|_{\hs}^2=\sum_{\ell\in\N}\int_{\T}\eta(x)^2|\nabla \phi e_\ell (x)|_{\R^{2}}^2\d x$ .
\end{lemma}
\begin{proof}[Proof of Lemma \ref{LD}]
It\^o Formula writes for $\frac12|\eta \nabla u|_{L^2}^2$:
\[\frac{|\eta \nabla u(t)|_{L^2}^2}{2}-\frac{|\eta \nabla u(0)|^2_{L^2}}{2}-\int_0^t\ps{\eta \nabla u}{\eta u\times\circ\nabla \d W}{}
=\int_0^t\ps{\eta^2 \nabla u}{\nabla (\tt)}{}\d r\, ,\]
Moreover, we have the identity
$\int_0^t\ps{\eta \nabla u}{\eta u\times\circ\nabla \d W}{}
=\iint\sum_{\ell \in\N}\eta ^2|\nabla \phi e_\ell|_{\R^2}^2+\int_0^t\ps{\eta \nabla u}{\eta u\times\nabla \d W}{}$ 
(the computations are identical as that of \eqref{en:item_i}, replacing $\ps{\cdot}{\cdot }{}$ by $\ps{\eta \cdot }{\eta \cdot }{}$).
We obtain:
\begin{multline*}
\frac{|\eta \nabla u(t)|_{L^2}^2}{2}-  \frac{|\eta\nabla u(0)|^2_{L^2}}{2}-\int_0^t|\eta \nabla \phi |_{\hs}^2\d r-\int_0^t\ps{\eta^2 \nabla u}{u\times \nabla \d W}{} 
\\
=\int_0^t-\ps{2 \eta \nabla \eta\nabla u+\eta^2\Delta u}{\tt}{}\d r
\\
\leq\int_0^t\left(-|\eta \tt|_{L^2}^2+|\eta \tt|_{L^2}^2+\frac{C}{\varrho^2}|\nabla u|^2_{L^2}\right)\d r\, ,\end{multline*}
a.s., where we have used that $\tt\cdot \Delta u=|\tt|^2$.
This proves \eqref{formula-LocalEnergy}.
\end{proof}
\begin{proof}[Proof of Proposition \ref{positivity}]
We first prove that for all $\ell,k\in\N$,
then $\Ps(\zetas _{\ell,k}>0)=1$.
Fix $\ell,k\in\N$.
We observe that:
\begin{equation}\label{eval:P_kappa_lk}
\begin{aligned}
\Ps(\zetas _{\ell,k}>0)&=1-\Ps(\zetas _{\ell,k}=0)\\
&=1-\lim_{N\to\infty}\Ps(\zetas _{\ell,k}\leq1/N)\, .
\end{aligned}
\end{equation}
To show that $\Ps(\zetas _{\ell,k}\leq1/N)\to0$ as $N\to\infty$, 
we need to circumvent the presence of a supremum in the definition of $\zeta $, which is not well adapted for martingale inequalities. This can be done via a discretization method, which relies on the following geometrical fact.
\paragraph{Covering argument.}
There exist constants $C=C(\T)>0$, $\lambda=\lambda (\T) \in(1,3]$, 
and a sequence of integers $\{N_k,k\in\N\}$ with $\limsup_{k\to\infty}(\varrho_k)^2N_k\leq C$,
such that for all $k\in \N$, there are points $\{x_k^{1},x_k^2,\dots,x_k^{N_k}\}\subset \T$ fulfilling the property:
\begin{equation}\label{coveringArg}
``\text{For all}\enskip x\in\T\enskip\text{there exists}\enskip i\in\{1,\dots, N_k\}\enskip \text{with}\enskip  B(x,\varrho_k) \subset B(x_k^i,\lambda \varrho_k)\,.''
\end{equation}
\item[\indent\textit{Proof.}]
It suffices to take $\lambda =2$, and to consider any finite cover $\T=\cup_{i\leq N_k}B(x_i^k,\varrho_k)$. Then we have also $\T=\cup_{i\leq N_k}B(x_i^k,2\varrho_k)$, and this cover fulfills the required property. Indeed, any ball $B(x,\varrho_k)$ is included in $B(x_i^k,2\varrho_k)$ whenever $|x-x_i^k|<\varrho_k,$ but such an $x_i^k$ always exists by assumption. This proves the covering argument.\hfill\qed
\vspace{1em}

Now, for each $k\in\N$, and each $x^i_k$,
consider $\eta=\eta _{\lambda \varrho_k,i}\in C^\infty_0(\T)$ with $\spt\eta\subset B(x^i_k,2\lambda \varrho_k)$
with
\begin{equation}\label{nota:eta-ky}
\mathds{1}_{B(\lambda \varrho_k,x_k^i)}\leq \eta\, ,\quad \sup\limits_{x\in B(x^i_k,2\lambda \varrho_k)}|\nabla \eta(x)|\leq \frac{C}{\varrho_k}\, ,
\end{equation}
for some $C>0$ independent of $i,k$.
To lighten the notations, denote by
\[
\cov{k}:=\{\eta\equiv\eta _{\lambda \varrho_k,i}\,,\,1\leq i\leq N_k \}\,,
\]
where the functions $\eta_{\lambda \varrho_k,i}$ are as above, so that in particular
$\card\cov{k}\equiv N_k$ is finite.
Using the bound on the local dissipation, namely \eqref{formula-LocalEnergy},
we have for all $\eta\in\cov{k}$:
\begin{equation}\label{local_dissipation}
\frac{1}{2}\left(|\eta \nabla \us_{\ell ,k}(t)|_{L^2}^2-|\eta \nabla \us_{\ell ,k}(0)|_{L^2}^2\right)
\leq V_{\ell ,k}^{\eta}(t)\, ,\enskip \text{for}\enskip t\in[0,\zeta _{\ell ,k}]\,,
\end{equation}
where we denote by:
$V_{\ell ,k}^{\eta}(t):=t|\eta\nabla \phi _{n_\ell} |_{\hs}^2+C/\varrho_k^2\int_0^t|\nabla \us_{\ell ,k}|_{L^2}^2\d r
+\int_0^t\ps{\eta\nabla \us_{\ell ,k}}{\eta \us_{\ell ,k}\times\nabla \d \Ws_\ell}{}.$

Moreover, the Burkholder-Davies-Gundy inequality for $\int_0^t\langle{\eta\nabla \us_{\ell ,k},\eta \us_{\ell ,k}\times\nabla \d \Ws_\ell}\rangle$ gives
\begin{multline}\label{ineq:localControlEn}
\Es\left[\sup_{0\leq t\leq 1/N}V^{\eta}_{\ell ,k}(t)\right]\leq
\frac{|\eta\nabla\phi _{n_\ell}|_{\hs}^2}{N}+\frac{C}{\varrho_k^2}\Es\left[\int_0^{\frac{1}{N}}|\nabla \us_{\ell ,k}|_{L^2}^2\d r\right]
\\
+C(|\phi |_{\hs^{1}})\Es\left[\int_0^{\frac{1}{N}}|\eta\nabla \us_{\ell ,k}|_{L^2}^2\d r\right]^{1/2}.
\end{multline}
On the other hand, according to the definition \eqref{zeta_tightness}, we have
\begin{multline}\label{strict}
\left\{\zeta_{n,k} \leq \frac1N\right\}\subset\left\{\enskip \zeta_{n,k} <\tau_n\enskip\text{and}\enskip \zeta_{n,k} \leq\frac1N\right\}
\cup\left\{\zeta_{n,k}=\tau_n\enskip \text{and}\enskip \tau_n\leq \frac1N \right\}
\\
=:\Omega_1\cup\Omega_2\,,
\end{multline}
but thanks to the bootstrap argument, namely Corollary \ref{cor:bootstrap}, we know that
\begin{equation}\label{P_zero_boot}
\P(\Omega_2)=0\,.
\end{equation}
Therefore, by \eqref{local_dissipation}, \eqref{k_geq_k0} and \eqref{ineq:localControlEn},
we see that
$\big\{|\eta \nabla \us_{\ell ,k}|_{L^2}^2\geq \epsilon _1\big\}\subset \big\{V^\eta _{\ell ,k}\geq \epsilon _1/4\big\}$
so that using on the other hand \eqref{coveringArg}, and Markov inequality, we obtain
\begin{equation}\label{ineq:p_N_lk}
\begin{aligned}
\Ps\left(\zetas _{\ell ,k}\leq\frac1N\right)&=\Ps\left(\sup\limits_{t\in[0,1/N]}\sup\limits_{x\in \T}\int_{B(x,\varrho_k)}|\nabla \us_{\ell ,k}(t)|^2\geq \epsilon _1\right)
\\
&\leq\sum_{\eta\in\cov{k}}\Ps\left(\sup_{t\in[0,1/N]} V_{\ell ,k}^{\eta}(t)\geq \frac{\epsilon _1}{4}\right)
\\
&\leq\frac{4}{\epsilon _1}\sum_{\eta\in\cov{k}}\Bigg\{\frac{|\eta\nabla\phi _{n_\ell }|_{\hs}^2}{N}+\frac{C}{\varrho_k^2}\Es\left[\int_0^{\frac{1}{N}}|\nabla \us_{\ell ,k}(r)|_{L^2}^2\d r\right]
\\
&\hspace{7em}+C(|\phi |_{\hs^{1}})\Es\left[\int_0^{\frac{1}{N}}|\eta\nabla \us_{\ell ,k}(r)|_{L^2}^2\d r\right]^{1/2}\Bigg\}\, .
\end{aligned}
\end{equation}
By the previous paragraph, the right hand side of \eqref{ineq:p_N_lk} converges 
to $0$ as $N\to\infty$, and the convergence holds
\textit{uniformly in $\ell \in\N$}.

\item[\indent\textit{Conclusion.}]
Writing that for each $\ell \in\N$:
\[\left\{\zetas _k\leq \frac1N\right\}\subset\left\{|\zetas _{\ell ,k}-|\zetas _{\ell ,k}-\zetas _k||\leq\frac{1}{N}\right\}\, ,\]
so that:
\[
 \Ps\left(\zetas _k\leq \frac{1}{N}\right)\leq \Ps\left(\zetas _{\ell ,k}\leq\frac{2}{N}\right)+\Ps\left(|\zetas _{\ell ,k}-\zetas _k|\geq \frac{1}{N}\right)\, .
\]
The conclusion follows by $|\zetas _{\ell ,k}-\zetas _k|\overset{\Ps}{\rightarrow}0$ as $\ell \to\infty$,
and the uniform convergence of $\Ps(\zetas _{\ell ,k}\leq1/N)$ as $N\to\infty$.
\end{proof}
\subsection{Step 4: uniqueness and the end of the Proof}\label{sub:gronwall}
We start by showing a useful Gronwall estimate on the difference of two martingale solutions $u,v$ of \eqref{LLG-Torus} that are defined on a common stochastic basis $\PPs\equiv(\Omegas,\Fs,\Ps,$ $(\Fs_t)_{t\in[0,T]};\Ws)$, and both supported in $C(0,\zeta;H^1)\cap L^2(0,\zeta;H^2)$ for some $\zeta >0$.
Namely, denoting by $f:=u-v$, we have
\begin{equation}\label{ineq:Gronwall}
\frac12|f(t)|_{L^2}^2\leq C\int_0^t(|\nabla u|_{L^4}^4+|\nabla v|_{L^4}^4+1)|f(r)|_{L^2}^2\d r\, ,\enskip \Ps\text{-a.s.\ for}\enskip t\in[0,\zeta ]\,.
\end{equation}
\begin{proof}
We have $f\in C(0,\zeta;H^1)$, $f(0)=0$ and 
\begin{equation}\label{eq:w_Ito}
\d f=\Big(\Delta f+u|\nabla u|^2-v|\nabla v|^2\Big)\d t+f\times\circ\d W' \, ,\enskip \text{on}\enskip \Omegas\times[0,\zeta ]\times\T\, .
\end{equation}
It\^o Formula on $\frac12|f|_{L^2}^2$ gives a.s.\
\begin{equation}\label{eq:d_delta}
\begin{aligned}
\d \left(\frac{|f|_{L^2}^2}{2}\right)-\underbrace{\ps{f}{ f\times\circ\d\Ws }{}}_{=0}&=\langle{f\,,\,\Delta f+u|\nabla u|^2 -v|\nabla v|^2 )}\rangle \d t\\
&=\big(-|\nabla f|_{L^2}^2+ \langle{f,u|\nabla u|^2 -v|\nabla v|^2 }\rangle\big)\d t\, .\\
\end{aligned}
\end{equation}
Using H\"older Inequality, the second term in the right hand side of \eqref{eq:d_delta} is estimated as
\[\begin{aligned}
\int_0^t\langle{f,u|\nabla u|^2-v|\nabla v|^2}\rangle \d r &\leq\int_0^t\big(|f|_{L^4}^2|\nabla u|_{L^4}^2+|f|_{L^4}|\nabla u+\nabla v|_{L^4}|\nabla f|_{L^2}\big)\d r\\
&\leq C\int_0^t(|\nabla u|_{L^4}^2+|\nabla v|_{L^4}^2)|f|_{L^4}^2 \d r+\frac12\int_0^t|\nabla f|_{L^2}^2\d r\, ,
\end{aligned}\]
a.s.\ for $t\in[0,\zeta ]$.
Since by Proposition \ref{pro:interp1} $\n{f}{L^4}^2\leq \mu _0(|\nabla f|_{L^2}+|f|_{L^2})|f|_{L^2}$, using again $ab\leq a^2/2+b^2/2$ yields:
\begin{equation}\label{ineq_II}
\int_0^t\langle{f,u|\nabla u|^2-v|\nabla v|^2}\rangle \d r\leq C\int_0^t(|\nabla u|_{L^4}^4+|\nabla v|_{L^4}^4+1)|f|_{L^2}^2\d r+\int_0^t|\nabla f|_{L^2}^2\d r\, .
\end{equation}
Putting together \eqref{eq:d_delta} and \eqref{ineq_II}, we obtain \eqref{ineq:Gronwall}.
Note that all computations above make sense since $u,v\in C(0,\zeta;H^1)\cap L^2(0,\zeta;H^2)\hookrightarrow L^4(0,\zeta;W^{1,4})$, by \eqref{ineq:interp}.
\end{proof}
\begin{corollary}
For any $u_0\in H^1,$ and $\phi \in \hs^{1},$ there exists a local strong solution $(u_*;0,\zeta _*)$ for \eqref{LLG-Torus}.
\end{corollary}
\begin{proof}
We use the famous Gy\"ongy and Krylov argument \cite{gyongy1996existence} (see also \cite{yamada1971uniqueness} and \cite{rockner2008yamada} for related results).
If we consider another extraction $\{m_\ell ,\ell \in\N\}$, then it is straighforward that the sequence
\begin{equation}\label{uplet}
\Big\{\left(U_{n_\ell },Z_{n_\ell},W_{n_\ell };U_{m_\ell },Z_{m_\ell},W_{m_\ell }\right)\,,\,\ell \in\N\Big\}\, ,
\end{equation}
is tight in $\mathfrak X:=\left(E\times(\prod_{k\in\N}[0,T])\times C^\alpha (0,T;H^1)\right)^2,$ where $\alpha <1/2,$ hence it is not restrictive to assume the existence of another sequence $\big(\{\uss_{\ell,k}\}_{k\in\N} ,\{\zetass_{\ell,k}\}_{k\in\N} , \Wss_\ell  \big),\ell \geq 0,$ as well as random variables $\{\uss_k\}_{k\in\N} ,\{\zetass_k\}_{k\in\N}$ on $\Omegas$, such that the conclusions of Corollary \ref{cor:conv} hold with $\uss$ instead of $\us$.

Fixing $k\in\N$, by \eqref{cv1}--\eqref{cv4} it is straightforward to show that the limits $\us_k,\uss_k$ are martingale solutions on $[0,\zetas _k]$, resp.\ $[0,\zetass_k]$. Moreover, they are both supported in $C(0,\kappa _k;H^1)\cap L^2(0,\kappa _k;H^2)$
where $\kappa _k:=\min(\zetas_k,\zetass_k)>0$. It follows from relation \eqref{ineq:Gronwall} that $\us_k|_{[0,\kappa _k]}=\uss_k|_{[0,\kappa _k]}$, and by reiteration we have also $\zetas_k=\zetass_k$, so that the weak limit of the sequence defined in \eqref{uplet} is supported in the diagonal of $\mathfrak X$.
This gives in particular the convergence of the whole sequence $(u_{n,k},\zeta _{n,k})_{n\in\N}$ towards a strong solution
$u_k:\Omega\times[0,\zeta _k]\to L^2$.

\item[\indent\textit{Definition of $(u_*,\zeta _*)$.}]
The definition \eqref{nota:zeta} implies that
$\zeta _{n,k}\leq\zeta _{n,k+1}$, $\P$-a.s., $\forall n,k\in\N$.
We can take the limit as $n\to\infty$, so that for each $k$:
\begin{equation}\label{zeta_k_decr}
\zeta _{k}\leq\zeta _{k+1}\enskip \P\text{-a.s.\null,}
\end{equation}
and the following definition is not ambiguous:
\begin{equation}\label{u_star}
u_*(u_0)(t):=\begin{cases}
          u_{k}(t)\enskip \text{if}\enskip t\in[0,\zeta _{k})\enskip \text{for some}\enskip k\geq 1\\
          0\enskip \text{otherwise.}
         \end{cases}
\end{equation}
This defines a local strong solution $\left(u_*(u_0);0,\zeta _*(u_0)\right)$, where we let
\begin{equation}
\label{zeta_star}
\zeta_*(u_0):=\sup_{k\in\N}\zeta _k\,.
\qedhere
\end{equation}
\end{proof}
\begin{proof}[End of the Proof of Theorem \ref{thm:struwe_sol}.]
It remains to show \eqref{thm_1}, \eqref{thm_2}, \eqref{thm_4} and \eqref{thm_5}.
We will proceed through successive steps.

\item[\indent\textit{Step 1. Proof of \eqref{thm_2}.}]
We show existence and uniqueness for the limit of $\{f_k:=u_{k}(\zeta _{k})\,,\,k\in\N\}$ in $L^2(\Omega;H^1)$-weak.
For $k,p\in\N$, using the equation on $u_k$ and $u_{k+p}$ gives:
\begin{multline}\label{ineq:fk}
\E\left[|f_{k+p}-f_{k}|_{L^2}^2\right]
\leq C\,\E\left[\int_{\zeta _{k}}^{\zeta_{k+p}}\Big|\Delta u_{k+p}
+u_{k+p}|\nabla u_{k+p}|^2+F_\phi u_{k+p}\Big|^2_{L^2}\d t\right]
\\
+C\left(|\phi|_{\hs(L^2)}\right)\E\left[\int_{\zeta _{k}}^{\zeta _{k+p}}|u_{k+p}|_{L^\infty}^2\d t\right]
\end{multline}
Since the sequence $\{\zeta _k\}$ is monotone and bounded, by \eqref{zeta_k_decr} we have $\aslim_{k\to\infty}|\zeta _{k+p}-\zeta _{k}|=0$. Therefore,
using \eqref{en:item_ii}, $|u_{n,k}|=1$ a.e.\ and \eqref{ineq:fk} gives that $(f_k)_{k\in\N}$ is a Cauchy sequence in $L^2(\Omega \times\T)$. Its limit $f=f(\omega ,x)$ is in $L^2(\Omega;H^1)$ by Prop.\ \ref{pro:enDecay}.

To prove \eqref{thm_1} and \eqref{thm_4}, we first need to establish the fact that the singular points are finite, $\P$-a.s.
We show in addition that during blow-up the solution releases a quantum of energy. This will be used in the proof of \eqref{thm_1}.

\item[\indent\textit{Step 2. Finiteness of the singular set and the proof of \eqref{thm_4}-\eqref{thm_5}.}]
Denote by $u:=u_*(u_0)$ and by
\begin{equation}\label{nota:Sg}
\sing(f)=\left\{x\in \T,\,\exists x_k\to x\,,\,|\nabla u(\zeta _{k})|_{L^2(B(x_k,\varrho_k))}^2\geq \epsilon _1\text{ for all }k\right\}\, .
\end{equation}
Using the definition of $\zeta _k$, for every family $(x^i)_{i\in I}\in\sing(\vartheta)^{I}$ of distinct elements, for every $i\in I,$
there exist $x^i_k\to x^i$ with $\int_{B(x^i_k,\varrho_k)}|\nabla u(\zeta _k)|^2\geq \epsilon _1$.
By semicontinuity of the norm with respect to weak convergence, for any $k\in\N$ large enough, we have:
\begin{multline}\label{loss}
|\nabla f|_{L^2(\T\setminus \cup_{i\in I} B(x_k^i,\varrho_k))}^2\leq \liminf_{p\to\infty}|\nabla f_p|_{L^2(\T\setminus \cup_{i\in I}B(x_k^i,\varrho_k))}^2\\
\leq |\nabla f_k|_{L^2(\T\setminus B(x_k,\varrho_k))}^2=|\nabla f_k|_{L^2(\T)}^2 - \sum_{i\in I}|\nabla f_k|_{L^2(B(x_k^i,\varrho_k))}^2
\end{multline}
(we can assume without restriction that the balls $B(x^i_k,\varrho_k)$ are disjoint since $x^i\neq x^j$ for $i\neq j$).
The right hand side in \eqref{loss} is bounded by $\n{u(\zeta _k)}{L^2(\T)}^2-\left(\card I\right)\epsilon _1$, and this holds for any $k\in\N$. 
Taking the limit in \eqref{loss} gives then
\begin{equation}\label{loss2}
\n{\nabla f}{L^2(\T)}\equiv\liminf_{k\to\infty}|\nabla f|_{L^2(\T\setminus B(x_k,\varrho_k))}^2\leq\liminf_{k\in\N}\n{\nabla f_k}{L^2(\T)}^2-(\card I)\epsilon _1\, .
\end{equation}
This implies in particular $\card\sing(f)<\infty$.
The properties \eqref{thm_4} and \eqref{thm_5} follow.

\item[\indent\textit{Step 3. Definition of the maximal solution.}] For $m\in\N^*$, define a measurable process $u:\Omega\times[0,T]\to H^1$, 
and a stopping time $\vartheta^m$
recursively by letting $\big(u|_{[0,\vartheta^1)};0,\vartheta^1\big)$ be $\left(u_*(u_0);0,\zeta _*(u_0)\right)$ i.e.\ the solution defined by \eqref{u_star}-\eqref{zeta_star}, and whenever $m\geq 1$:
\begin{equation}\label{recursion}
\left[\begin{aligned}
&u(\vartheta^{m}):=\lim_{t\nearrow\vartheta^{m-1}} u_*(u^{m-1})(t)\enskip \text{in}\enskip L^2(\Omega;H^1)\enskip \text{weak}\,,\\[0.6em]
&\vartheta^{m+1}:=\vartheta^m+\zeta _*(u(\vartheta^{m}))\, ,\\[0.6em]
&u|_{[\vartheta^{m},\vartheta^{m+1})}(t -\vartheta^m):=u_*(u(\vartheta ^m))(t)\,,\enskip t\in[\vartheta^{m},\vartheta^{m+1})\, .
\end{aligned}\right.
\end{equation}
This procedure can be repeated by the fact that the limit $f$ is in $L^2(\Omega;H^1)$ and is measurable with respect to $\mathcal F_{\zeta _*}$.

\item[\indent\textit{Step 4. Proof of \eqref{thm_1}.}]
To prove that the solution constructed above is global, we define the $\N\cup\{\infty\}$-valued process
\begin{equation}\label{nota:Nt}
N_t:=\begin{cases}
\card\left\{(x,s)\in\T\times[0,t)\,,\,\inf_{\epsilon,\varrho \searrow0}\int_{B(x,\varrho)}|\nabla u(s-\epsilon,y )|^2\d y>0\right\}\, ,\\
\hspace{3em}\text{if}\enskip t\leq\sup_{m\in\N}\vartheta^m\, ,\\[1em]
\infty\qquad \text{if}\enskip t\in(\sup_{m\in\N}\vartheta^m,T]\,,
     \end{cases}
\end{equation}
so that
\begin{equation}\label{bound:PTheta}
\P(\forall m\in\N,\enskip \vartheta ^m<T)\leq\P(N_T=\infty)\, .
\end{equation}
Using \eqref{loss2} together with Proposition \ref{pro:enDecay},
we see that
\[
 \E \left[E_{\vartheta^1}\right]\leq \lim_{k\to\infty}\E \left[E_{\zeta _k}-\frac{\epsilon _1}{2}N_{\vartheta^1}\right]
\leq \E\left[ E(0)\right]+C(\n{\phi }{\hs^1})\E\left[\vartheta^1\right]-\frac{\epsilon _1}{2}\E \left[N_{\vartheta^1}\right]\, ,
\]
and a straightforward induction implies that for each $t\in[\vartheta^m,\vartheta^{m+1})$:
$\E [E_{\vartheta^m}]\leq E(0) +C(\n{\phi}{\hs^1})t -\epsilon _1/2\E[N_t]$,
which finally gives the bound:
\begin{equation}\label{bound:ENt}
\E \left[N_T\right]\leq \frac{2}{\epsilon _1}\left(E(0)+C(|\phi|_{\hs^1})T\right)\, .
\end{equation}
The conclusion now follows from
\eqref{bound:PTheta} and \eqref{bound:ENt}:
we have $\P(\forall m\in\N,\enskip \vartheta ^m<T)=0$, and thus
$\P(\exists m\in\N,\enskip \vartheta ^m=T)=1$.
This finishes the proof of Theorem \ref{thm:struwe_sol}.
\end{proof}
\section{Proof of Theorem \ref{thm:uniqueness}}
\label{sec:proof_thm2}
\subsection{Treatment of the regular part of the solution}
\label{subsec:regular_part}
Let $(u,\PP)$, denote a martingale solution in the sense of Definition \ref{def:martingale_solution}.
In order to prove theorem \ref{thm:uniqueness}, we aim to
decompose $u$ into $\uu +\nu$, where $\nu $ is the ``singular part''. We first need to isolate the term in $u|\nabla u|^2$ that corresponds to possible degeneracies. Using that $u\cdot \nabla u= 0$, Helein's decomposition writes for $i=1,2,3$:
\begin{equation}\label{decomp:helein}
\begin{aligned}
u^i|\nabla u|^2&=\sum_{1\leq j\leq3,1\leq k\leq2}(u^i\partial_ku^j -u^j \partial_ku^i)\partial_ku^j\\
&=\sum_{1\leq j\leq3,1\leq k\leq2}A^{i,j}_k\partial_ku^j\equiv A\cddot\nabla u^i\, ,
\end{aligned}
\end{equation}
where from now on the double dots $X\cddot f $ will be used to denote the ``collapse of the $(k,j)$ indices'' of two tensors $(X^{i,j}_k)\in(\R^3)^{\otimes2}\otimes\R^2$ and $(f^{j}_k)\in(\R^3)^{\otimes2},$ namely
\[
X\cddot f:=\Big(\sum_{1\leq j\leq3,1\leq k\leq2}X^{i,j}_kf^j_k\Big)_{1\leq i\leq 3}\,.
\]

We recall the following classical theorem for the decomposition of two-dimensional vector fields. The following version can be found in \cite{dautray2012mathematical}, as a consequence of 
Prop.\ 1 p.\ 215, and Prop.\ 3 p.\ 222.
\begin{theorem}[Helmholtz]\label{thm:Helmholtz}
We have the orthogonal decomposition:
\begin{equation}
L^2\left(\T;(\R^3)^{\otimes2}\otimes\R^2\right)=\nabla H^1\left(\T,(\R^3)^{\otimes2}\right)\oplus\nabla^\perp H^1\left(\T;(\R^3)^{\otimes2}\right)\,.
\end{equation}
The corresponding projections are continuous in $L^2$.
\end{theorem}

Applying Theorem \ref{thm:Helmholtz}, we write for each $t\in[0,T]$:
\[
A(t)=\nabla \alpha (t)+\nabla ^\perp\beta(t)\, ,
\]
where $A(t)$ is defined by \eqref{decomp:helein} with $u\equiv u(t)$ being the trace of $u$ onto $\{t\}\times\T$.
Taking the divergence, we obtain for each $1\leq i,j\leq3$:
\begin{equation}\label{eq:div}
\div A^{i,j}=u^i\Delta u^j-u^j\Delta u^i\, ,
\end{equation}
and since $\n{u\transp\Delta u-\Delta u\transp u}{(\R^3)^{\otimes2}}^2=\n{\tt}{\R^3}^2$, we have $\nn{\div A}{L^2(0,T;L^2)}\leq C\|\tt\|_{L^2(0,T;L^2)},$
where we define $\tt$ as in \eqref{nota:tension}.
On the other hand since $\div A=\Delta\alpha,$
we obtain that
\begin{equation}\label{Pathwise}
\|\Delta \alpha \|_{L^2([0,T]\times\T)}\leq C\nn{\tt}{L^2(0,T;L^2)}\, ,
\end{equation}
$\P$-a.s.
Consider now the equation (with unknown $\uu$):
\begin{equation}\label{eq:f_ito}
\left\{\begin{aligned}
   &\d \uu-\Delta \uu\d t=\nabla \alpha\cddot \nabla u\d t+u\times \circ\d W\, ,\enskip \text{on}\enskip \Omega\times[0,T]\times\T\, ,\\[2mm]
   &\uu(0)=u_0\,,\hspace{1.2em}\enskip \text{on}\enskip\Omega\times\T\, .\\[1mm]
\end{aligned}\right.
\end{equation}
Note that $\uu$ solves \eqref{eq:f_ito} in the sense of distributions
if and only if
\begin{equation}\label{decomp:u_hat}
\uu=u^{\flat}+u^{\sharp}+Z\,,
\end{equation}
where respectively
\begin{empheq}{align}
\label{eq:u_flat}
\partial _tu^{\flat}-\Delta u^{\flat}=0\,,\quad u^{\flat}(0)=u_0\,,\\
\label{eq:u_sharp}
\partial _tu^{\sharp}-\Delta u^{\sharp}=\nabla\alpha \cddot\nabla u \,,\quad u^{\sharp}(0)=0\,,\\[0.25em]
\label{eq:Z}
\d Z=\Delta Z\d t+u\times\circ\d W\,,\quad Z(0)=0\,.
\end{empheq}
From \eqref{eq:f_ito}, we can now deduce better regularity for $\uu$, namely:
\begin{claim}\label{clm:1}
With probability one, there exists a unique solution $\uu$ of \eqref{eq:f_ito} in $L^4(0,T;W^{1,4})$.
\end{claim}
\begin{proof}
By Proposition \ref{pro:parab_sto}, there exists a unique weak solution $Z$ to \eqref{eq:Z}, given by
$Z(t)=\int_0^tS(t-r)F_\phi u\d r+\int_0^tS(t-r)u\times\d W,$
moreover we have:
\begin{multline}
\E\left[\|Z\|_{L^4(W^{1,4})}^4\right]\leq C\E\left[\|Z\|_{L^4(H^{3/2})}^4\right]
\leq C' T^4\E\left[\sup_{0\leq t\leq T}|u(t)\times\phi |_{\hs^1}^4\right]
\\
\leq C''T^4(1+|\nabla \phi |_{\hs}^4),
\end{multline}
by the cancellations occuring in \eqref{hs1}.
Moreover, by Proposition \ref{pro:parab} and $L^{4/3}\hookrightarrow W^{-1,4}:$
\[
\big\|\int_0^\cdot S(\cdot -r)F_\phi u\d r\big\|_{L^4(W^{1,4})}
\leq C\|F_\phi u\|_{L^4(L^{4/3})}
\leq CT\sum_{\ell \in\N} |\phi e_\ell |^2_{L^{8/3}}\leq C'T|\phi |^2_{\hs^{1}}\,.
\]
Therefore
\begin{equation}\label{Z_reg}
Z(\omega )\in L^4(0,T;W^{1,4})\, ,\enskip \text{a.s.}
\end{equation}
On the other hand,
H\"older Inequality,
\eqref{ineq:interp} and \eqref{Pathwise} yield
\begin{equation}\label{bound:f}
\nn{\nabla \alpha\cddot \nabla u}{L^4(L^{4/3})}\leq \nn{\nabla \alpha }{L^4(L^4)}\nn{\nabla u}{L^\infty (L^2)}\leq \|\tt\|_{L^2(L^2)}\|\nabla u\|_{L^\infty (L^2)}\, .
\end{equation}

Using again $L^{4/3}\hookrightarrow W^{-1,4}$ we also have that $\nabla\alpha\cddot \nabla u\in L^4(0,T;W^{-1,4})$,
so that by Proposition \ref{pro:parab} there exists a unique $u^{\sharp}\equiv\mathscr V(\nabla \alpha\cddot\nabla u)\in L^4(0,T;W^{1,4})$ such that \eqref{eq:u_sharp} holds.

Finally, Proposition \ref{pro:parab} yields existence and uniqueness of $u^{\flat}\in C(0,T;H^1)\cap L^2(0,T;H^2)$ solving \eqref{eq:u_flat}, which by \eqref{ineq:interp} also belongs to $L^4(0,T;W^{1,4})$.
This proves Claim \ref{clm:1}.
\end{proof}
\subsection{Decomposition of ``$\nabla ^\perp\beta$''.}\label{subsec:decomposition_beta}
The previous paragraph shows that the symmetric part of $A^{i,j}\equiv u^i\nabla u^j-u^j\nabla u^i$ is controlled by the bound on the tension $\tt$. This yields $L^4(W^{1,4})$-regularity for the renormalized solution $\uu$. Oppositely, the antisymmetric part $\nbp\beta \cddot\nabla u$,
can be singular, at least without further assumptions on $u$.
However, using the identity $|A|^2=2|\nabla u|^2$ one can write for any $0\leq s\leq t\leq T:$
\begin{equation}\label{rel:G_beta}
\frac14|A(t)|^2_{L^2}\equiv\frac14\left(|\nabla \alpha (t)|_{L^2}^2 + |\nbp\beta (t)|^2_{L^2}\right)=E(t)=\mathscr G(t)+C_\phi t\,,
\end{equation}
so that additional regularity will be provided by proving local continuity for $t\mapsto\mathscr G(t)$.
This follows from the supermartingale property.
\begin{claim}\label{clm:2}
With full probability, $t\in[0,T]\mapsto|\nbp \beta (t)|_{L^2}^2$ is right-continuous.
\end{claim}
\begin{proof}
Let $s\in[0,T]\,$.
Define for $p,n\in\N\,$, the set
$\mathcal{U}(p,n)=\{\omega \in\Omega:\exists t_n(\omega)\in[s,s+(n+1)^{-1}],\enskip |\mathscr{G}(t_n)-\mathscr{G}(s)| >(p+1)^{-1}\}$.
It is convenient to write that
\begin{equation}\label{right_cont}
\left\{\omega:t\mapsto\mathscr G(t)\enskip \text{is not right-continuous at}\enskip  t=s\right\}
=\cup_{p\in\N}\cap_{n\in \N}\mathcal{U}(p,n)\,.
\end{equation}

\item[\indent\textit{Right continuity of $\mathscr G$.}]
Reasoning by contradiction, assume that there exists $p\in\N$ such that $\Omega_p\equiv\cap_{n\in\N}\mathcal{U}(p,n)$
has positive probability.
The Optional Sampling Theorem (see e.g.\ \cite[Chap.~I-6]{watanabe1981stochastic}),
implies
\begin{equation}\label{supermartingale_property_tn}
\E_{\mathcal{F}_s}[\mathscr{G}(t_n)-\mathscr{G}(s)]\leq0\enskip \text{a.s.\ }
\end{equation}
Moreover, classical facts on supermartingales (see e.g.\ \cite[Thm.\ 6.8]{watanabe1981stochastic}) imply the existence of
$\mathscr{\tilde G}(s):=\aslim_{n\to\infty}\mathscr{G}(t_n)$.
Note that by the right-continuity assumption on $(\mathcal{F}_t)\,$,
the set $\Omega_p$ is $\mathcal{F}_s$-measurable.
On the one's hand, there holds
\begin{equation}\label{superMart1}\begin{aligned}
\E_{\mathcal{F}_s}[(\mathscr{\tilde G}(s)-\mathscr{G}(s))\mathds{1}_{\Omega_p}]
&=\E\Big[\E_{\mathcal{F}_s}\big[(\mathscr{\tilde G}(s) -\mathscr{G}(s))\mathds{1}_{\Omega_p}\big]\Big]\\
&\leq\E\Big[\mathds{1}_{\Omega_p}\liminf_{n\to\infty}\E_{\mathcal{F}_s}[\mathscr{G}(t_n)-\mathscr{G}(s)]\Big]
\\
&\leq0\, ,\qquad\text{by \eqref{supermartingale_property_tn}.}
\end{aligned}\end{equation}
On the other hand $\P(u\in C(0,T;L^2))=1\,$, therefore
$\nabla u(t_n )\to\nabla u(s )$ a.s.\ in $H^{-1}\,$, and since $\nabla u(s )\in L^2(\Omega\times\T)$, we have in fact
\begin{equation}\label{weak_convergence_u_tn}
\nabla u(t_n )\to \nabla u(s)\enskip \text{weakly in}\enskip L^2(\T)\,,\enskip \text{a.s.}
\end{equation}
By lower semicontinuity of the $L^2$-norm, we have
$\mathscr{G}(s)\leq \mathscr{\tilde G}(s)=\lim \mathscr{G}(t_n)\,,$ a.s.
On the other hand, since on $\Omega_p$ we have $|\mathscr{G}(t_n)-\mathscr{G}(s)|>1/(p+1)$ for all $n\geq 0\,$, it follows that
\[\mathds{1}_{\Omega_p}\left(\mathscr{\tilde G}(s)-\mathscr{G}(s)\right)=\mathds{1}_{\Omega_p}|\mathscr{\tilde G}(s) - \mathscr{G}(s)|\geq\frac{\mathds{1}_{\Omega_p}}{p+1}\, .\]
This lower bound, together with \eqref{superMart1} and $\P(\Omega_p)>0\,$, leads to a contradiction.

The right-continuity of $\beta $ follows by \eqref{rel:G_beta}.
\end{proof}
\subsection{Conclusion}\label{subsec:conclusion_wente}
To end the proof of Theorem \ref{thm:uniqueness}, analogous arguments as that of the proof given in \cite{freire1995uniqueness} will be used, although the important difference here is that we do not refer to the Struwe solution. This gives in addition a new proof of A.\ Freire's Theorem.

Denote by
\begin{equation}\label{nota:nu}
\nu :=u-\uu\, ,
\end{equation}
and note that $\nu $ is a weak solution of 
\begin{equation}\label{eq:nu}
\partial _t\nu -\Delta \nu =\nbp\beta \cddot\nabla \uu  +\nbp\beta\cddot\nabla \nu\,,\quad\nu (0)=0\,.
\end{equation}
Let us point out that because of the cancellation that occurs in \eqref{nota:nu}, we make no use of any stochastic argument here. Therefore in the next computations we will simply omit the sample, assuming  $\omega \in\Omega\setminus\mathcal N$, $\mathcal N$ being a $\P$-null set such that $\uu(\omega )\in L^4(0,T;W^{1,4})$ for $\omega \notin \mathcal N$.

By density of $C^\infty_{t,x}$ in $C(0,T; H^1)\,$,
the right-continuity of $\beta $ yields that for all $\epsilon>0\,$ there exists $0<\tau(\epsilon)\leq T$ and 
$\beta_\epsilon\in L^\infty(0,\tau;H^1)\,$, $\beta_\epsilon'\in C^\infty([0,\tau ]\times\T)$ with:
\begin{equation}\label{decomp:beta}
\beta|_{[0,\tau]}=\beta_\epsilon+\beta_\epsilon'\enskip \text{and}\enskip \|\beta_\epsilon\|_{L^\infty(0,\tau;H^1)}\leq\epsilon\, .
\end{equation}
Choosing $\epsilon\,$, $\tau>0$ as in \eqref{decomp:beta},
we let
$g_\epsilon:=\left(\nbp\beta \cddot\nabla \uu+\nabla^\perp\beta_\epsilon'\cddot\nabla \nu \right)|_{[0,\tau ]}$.
Using again the abbreviation $L^p(L^q):=L^p(0,\tau;L^q),$
immediate computations yield that $g_\epsilon \in L^4(W^{-1,4})$, since $L^{4/3}\hookrightarrow W^{-1,4}$ and
\begin{equation}\label{bound:g_eps}
\|g_\epsilon \|_{L^4(L^{4/3})}\leq \|\nbp\beta \|_{L^\infty( L^2)}\|\nabla \uu\|_{L^4(L^4)} + C \|\nabla \nu\|_{L^\infty( L^2)}<\infty\,,
\end{equation}
where the bound on $\uu$ is justified by the improved regularity shown in Paragraph \ref{subsec:regular_part}.

In the sequel, we will denote by
$\mathscr U$ and $\mathscr V$ the bounded isomorphisms given by Proposition \ref{pro:parab}, which are respective inverses of
\[
\begin{aligned}
&\mathscr \partial _t-\Delta :L^2(0,T;H^1)\cap H^1_0(0,T;H^{-1})\longrightarrow L^2(0,T;H^{-1})\\
&\partial _t-\Delta :L^4(0,T; W^{1,4})\cap W^{1,4}_0(0,T; W^{-1,4}) \longrightarrow L^4(0,T;W^{-1,4})\,.
\end{aligned}
\]
In \eqref{eq:nu}, we replace now $\nu $ by the unknown $\Phi$ and write the corresponding equation first with the help of $\mathscr V$ as:
\begin{equation}\label{eq:mild_B}
\Phi-\mathscr V(\nabla^\perp\beta _\epsilon \cddot \nabla \Phi)=\mathscr Vg_\epsilon\,.
\end{equation}
Letting $T_\epsilon\Phi :=\mathscr V\nabla^\perp\beta _\epsilon\cddot\nabla \Phi$,
the parabolic estimates, and the continuous embedding $W^{1,4/3}\hookrightarrow L^4\,$ give
$\|T_\epsilon \Phi\|_{L^4(W^{1,4})}\lesssim \|\nabla^\perp\beta _\epsilon\cddot\nabla \Phi\|_{L^4(W^{-1,4})}\lesssim\|\nabla^\perp\beta _\epsilon\cddot\nabla \Phi\|_{L^4(L^{4/3})}.$
Using then H\"older Inequality, we have
\begin{equation}\label{estimate_V_epsilon}
\nn{T_\epsilon \Phi}{L^4(W^{1,4})}\leq C(T) \|\nabla^\perp\beta _\epsilon \|_{L^\infty (L^2)}\|\nabla \Phi\|_{L^4 (L^4)}\leq C(T)\epsilon\nn{\Phi}{L^4(W^{1,4})}\,,
\end{equation}
with a constant depending on $T>0$ but not on $\tau$ (because the operator norm of $\mathscr V$ increases with $\tau $).
Taking $\epsilon < C(T)^{-1}$ we have
$\|T_\epsilon\|_{\mathscr L(L^4(W^{1,4}))}<1$ yielding the convergence of the Neumann Series $\sum_{n\leq N}(T_\epsilon )^n\to(\mathrm{id}-T_\epsilon )^{-1}.$ This gives the existence of a (unique) $\Phi \in L^4(W^{1,4})$, solving \eqref{eq:nu}. However, we do not know at this stage whether $\Phi $ equals $\nu $.

Since the bound \eqref{bound:g_eps} yields also $g_\epsilon \in L^2(H^{-1})$,
the same reasoning as above, but with $\mathscr U:L^2(H^{-1})\to L^2(H^1)\cap H^1_0(H^{-1})$ instead of $\mathscr V$, leads to the equation
\begin{equation}\label{eq:mild_H}
\Psi-\mathscr U(\nabla^\perp\beta _\epsilon\cddot\nabla \Psi)=\mathscr Ug_\epsilon\, ,
\end{equation}
with unknown $\Psi$ in $L^2(H^1)$.
We will now make use of the following ``regularity by compensation'' result.
\begin{theorem}[\cite{wente1969existence}]\label{thm:Wente}
For $a ,b \in H^1(\T;\R)\,$, let 
$\varphi $ be the unique solution of
\[
\varphi+\Delta \varphi = \{a ,b \}
\]
on $\T$
where $\{a ,b \}$ denotes the Poisson bracket
$\partial _{1}a \partial _{2}b -\partial _{2}a \partial _{1}b \,$.
Then $\varphi \in C(\T;\R)\cap H^1(\T;\R)\,$, and
\begin{equation}\label{ineq:Wente}
|\varphi |_{L^\infty}+|\nabla \varphi |_{L^2}\leq C|\nabla a|_{L^2}|\nabla^\perp b |_{L^2}\, ,
\end{equation}
for a constant independent of $\varphi \,$.
\end{theorem}

Denoting by $\tilde T_\epsilon \Psi:=\mathscr U(\nabla ^\perp\beta _\epsilon\cddot\nabla \Psi),$ then the parabolic estimates, Theorem \ref{thm:Wente} and \eqref{decomp:beta}
give that for all $\Psi \in L^2(H^1)$:
\begin{equation}\label{bound_T_epsilon_0}
\|\tilde T_\epsilon \Psi\|_{L^2(H^1)}\leq C\|\{\beta _\epsilon ,\Psi\}\|_{L^2(H^{-1})}\leq C(T)\epsilon \|\Psi\|_{L^2(H^1)}\,.
\end{equation}
Assuming in addition $\epsilon <\min({C(T),C'(T)})$, the same argument as above yields \emph{uniqueness} of $\Psi$ within the class $L^2(H^1)$, solving \eqref{eq:mild_B}. Since we already know that $\nu$ belongs to this class (because $u$ does), and since $L^4(W^{1,4})\hookrightarrow L^2(H^1)$, we obtain that
\begin{equation}\label{egalite}
\begin{aligned}
\nu &= \Psi\\
&=\Phi
\end{aligned}
\end{equation}
and therefore 
\begin{equation}\label{u:W14}
u\equiv \uu+\nu \in L^4(0,\tau ;W^{1,4})\,,
\end{equation}
$\P$-a.e.

Finally, the Gronwall inequality shown in paragraph \ref{sub:gronwall} (see \eqref{ineq:Gronwall}) shows that $u$ is necessarily the Struwe solution constructed in Theorem \ref{thm:struwe_sol}.
Theorem \ref{thm:uniqueness} is now proved.
\hfill\qed

\section*{Acknowledgements}
This work was done in its majority while the Author was a Phd student at the \'Ecole polytechnique. Anne De Bouard and Francois Alouges are gratefully acknowledged for numerous discussions and encouragements.
Financial support was kindly provided by the ANR projects Micro-MANIP (ANR-08-BLAN-0199) and STOSYMAP (ANR-2011-BS01-015-03).



\end{document}